\documentclass{article}
\usepackage{import}
\usepackage[utf8]{inputenc}
\usepackage[full]{textcomp}
\usepackage[osf]{newtxtext}
\usepackage{comment}

\usepackage{amssymb}
\usepackage{mathtools}
\usepackage{hyperref}
\usepackage{breakurl}
\usepackage{mhenvs}
\usepackage{mhequ}
\usepackage{mhsymb}
\usepackage{booktabs}
\usepackage{tikz}
\usepackage{mathrsfs}
\usepackage{longtable}
\usepackage{times}
\usepackage {tikz}
\usetikzlibrary {positioning}

\newtheorem*{lemma**}{Lemma}
\newtheorem*{theorem**}{Theorem}

\usepackage{microtype}
\usepackage{wasysym}
\usepackage{centernot}
\usepackage{enumitem}

\numberwithin{equation}{section}

\makeatletter
\newcommand{\globalcolor}[1]{%
  \color{#1}\global\let\default@color\current@color
}
\makeatother

\newif\ifdark
\IfFileExists{dark}{\darktrue}{\darkfalse}

\ifdark

\definecolor{darkred}{rgb}{0.9,0.2,0.2}
\definecolor{darkblue}{rgb}{0.7,0.3,1}
\definecolor{darkgreen}{rgb}{0.1,0.9,0.1}
\definecolor{pagebackground}{rgb}{.15,.21,.18}
\definecolor{pageforeground}{rgb}{.84,.84,.85}
\pagecolor{pagebackground}
\AtBeginDocument{\globalcolor{pageforeground}}

\else

\definecolor{darkred}{rgb}{0.7,0.1,0.1}
\definecolor{darkblue}{rgb}{0.4,0.1,0.8}
\definecolor{darkgreen}{rgb}{0.1,0.7,0.1}
\definecolor{pagebackground}{rgb}{1,1,1}
\definecolor{pageforeground}{rgb}{0,0,0}
\definecolor {processblue}{cmyk}{0.96,0,0,0}

\fi

\theoremnumbering{Alph}
\renewtheorem{theorem*}{Theorem}

\DeclareMathAlphabet{\mathbbm}{U}{bbm}{m}{n}
\DeclareFontFamily{U}{BOONDOX-calo}{\skewchar\font=45 }
\DeclareFontShape{U}{BOONDOX-calo}{m}{n}{
  <-> s*[1.05] BOONDOX-r-calo}{}
\DeclareFontShape{U}{BOONDOX-calo}{b}{n}{
  <-> s*[1.05] BOONDOX-b-calo}{}
\DeclareMathAlphabet{\mcb}{U}{BOONDOX-calo}{m}{n}
\SetMathAlphabet{\mcb}{bold}{U}{BOONDOX-calo}{b}{n}

\let\epsilon\varepsilon
\def\E{{\symb E}}
\def\F{{\mathcal F}}

\def\FC{\mathscr{C}}

\def\X{\mathbb{X}}

\def\WW{{ \mathbb W}}
\def\X{{\mathbf X}}
\def\XX{{\mathbb X}}

\def\err{\mathbf {Er}}

\def\C{\mathcal{C}}
\def\F{\mathcal{F}}
\def\f{\frac}
\def\1{\mathbf{1}}

\def\${|\!|\!|}

\def\<{\langle}
\def\>{\rangle}
\setlist{noitemsep,topsep=4pt}
\def\para_#1{/\!\!/_{\!#1}}

\def\slash{\kern0.18em/\penalty\exhyphenpenalty\kern0.18em}
\def\dash{\kern0.18em--\penalty\exhyphenpenalty\kern0.18em}

\makeatletter 
\newcommand*{\fat}{}
\DeclareRobustCommand*{\fat}{%
\mathbin{\mathpalette\bigcdot@{}}}
\newcommand*{\bigcdot@scalefactor}{.5}
\newcommand*{\bigcdot@widthfactor}{1.15}
\newcommand*{\bigcdot@}[2]{%
  \sbox0{$#1\vcenter{}$}
  \sbox2{$#1\cdot\m@th$}%
  \hbox to \bigcdot@widthfactor\wd2{%
    \hfil
    \raise\ht0\hbox{%
      \scalebox{\bigcdot@scalefactor}{%
        \lower\ht0\hbox{$#1\bullet\m@th$}%
      }%
    }%
    \hfil
  }%
}
\makeatother
{\theorembodyfont{\rmfamily}

\newtheorem{convention}[lemma]{Convention}
}

\newtheorem{assumption}[lemma]{Assumption}

\begin{document}
\title{Functional Limit Theorems for power series with rapid decay of moving averages of Hermite processes}

\author{Johann Gehringer \\Imperial College London \footnote{johann.gehringer18@imperial.ac.uk}}

\maketitle

\begin{abstract}
We aim to generalize the homogenisation theorem in \cite{Gehringer-Li-tagged} for a passive tracer interacting with a fractional Gau{\ss}ian noise to also cover fractional non-Gau{\ss}ian noises. 
To do so we analyse limit theorems for normalized functionals of Hermite-Volterra processes, extending the result in \cite{Diu-Tran} to power series with fast decaying coefficients. We obtain either convergence to a Wiener process, in the short-range dependent case, or to a Hermite process, in the long-range dependent case. Furthermore, we prove convergence in the multivariate  case with both, short and long-range dependent components. 
Applying this theorem we  obtain a homogenisation result for a slow/fast system driven by such Hermite noises.
\end{abstract}

{ \scriptsize {\it  keywords:} Central limit theorems, fractional noise, Hermite Ornstein-Uhlenbeck process, Hermite processes, homogenization of fast/slow systems}


\setcounter{tocdepth}{2}
\tableofcontents
\section{Introduction}\label{section-introduction}
In \cite{Gehringer-Li-tagged} the following model for a passive tracer was considered, 
\begin{equation}\label{multi-scale}
\left\{ \begin{aligned} \dot x_t ^\epsilon &=\sum_{k=1}^N\alpha_k(\epsilon) \, f_k(x_t^\epsilon) \,G_k(y_t^\epsilon),\\
x_0^\epsilon&=x_0, \end{aligned}\right.  
\end{equation}
where the $\alpha_k(\epsilon)$'s denote suitable scalings, the usual diffusive scaling corresponds to $\alpha(\epsilon) = \f {1} {\sqrt \epsilon}$,  and $y^{\epsilon}_t$  the rescaled stationary  fractional Ornstein-Uhlenbeck process. It was shown that given functions $f_k$ and $G_k$ of suitable regularity, for details we refer to \cite{Gehringer-Li-tagged}, the solutions to equation (\ref{multi-scale}), $x^{\epsilon}_t$, converge weakly to a stochastic process $x_t$. Furthermore, the limiting process $x_t$ solves the following stochastic differential equation,
$$
x_t =  x_0 + \sum_{k=1}^N \int_0^t f(x_s) dX^k_s ,
$$
where $X^k_t = \lim_{\epsilon \to 0} \alpha_k(\epsilon) \int_0^t G_k(y^{\epsilon}_s) ds.$ In particular these limits exist, however, unlike in the case of diffusive homogenisation they are not necessarily given by  Wiener processes. In case the memory of the fractional Ornstein-Uhlenbeck process is too strong the resulting equation might also be driven by Hermite processes, $Z^{H,m}_t$. 

Hermite processes are a two-parameter family of self similar processes with stationary increments. They can be represented via iterated Wiener integrals, up to normalizing constants, as follows,
$$
Z_t^{H,m}= \int_{\R^m} \int_0^t \prod_{j=1}^m (s-\xi_j)_+^{  H_0 - \f 3 2 } \, ds \,  d  B_{\xi_1} \dots d  B_{\xi_m},
$$
hence, the rank $m$ processes belong to the $m^{th}$ Wiener chaos. In particular they are Gau{\ss}ian if and only if $m=1$ and  in this case the above formulae matches the Mandelbrot Van-Ness representation of a fractional Brownian motion, c.f. \cite{Mandelbrot-VanNess}. Moreover, Hermite processes appear as limits in so called non-central limit theorems, see \cite{Taqqu-75,Taqqu,Rosenblatt,Bai-Taqqu,Dobrushin,Dobrushin-Major,Gehringer-Li-fOU}.  For an application to financial modelling, see \cite{Hermite_Markets}. Wiener integrals with respect to them as well as the Hermite Ornstein-Uhlenbeck have been introduced in \cite{Maejima-Ciprian}. Moreover, they satisfy
$$ \E \left[ Z^{H,m}_t Z^{H,m}_s \right] = \f 1 2 \left(\vert t\vert^{2H} + \vert s\vert^{2H} - \vert t-s\vert^{2H} \right),$$
thus, all Hermite processes have the same covariance structure as fractional Brownian motions.

Therefore, it seems natural to also consider passive tracers in Hermite noise fields.
Our aim is to provide the foundation for a similar type of homogenisation theorem covering systems as equation (\ref{multi-scale}) when the  environment for the passive tracer is given by a moving average of a Hermite process, 
$$
y_t = \int_{-\infty}^{t} x(t-s) Z^{H,m}_s,
$$
for a suitable regular kernel $x$, see Assumption \ref{assumption-decay-correlation-kernel} below.
To do so we require central and non-central limit theorems for functionals of the form
$$
\alpha(\epsilon) \int_0^{\f T \epsilon} G(y_t) dt,$$
where $\alpha(\epsilon)$
again denotes a suitable scaling.
In case $G=X^2 $ it was shown in \cite{Diu-Tran}, that, after subtracting the average, the right scaling is given by $\alpha(\epsilon) = \epsilon^{2H_0 -1}$, where $H_0 = 1 + \f {H-1} {m}$, and the limiting process is given by a Rosenblatt process; a Hermite process for which $m=2$.
In this article we are extending this result to functions of the form $G(X)= \sum_{k=0}^{\infty} c_k X^k$, for which the coefficients satisfy $\vert c_k \vert \lesssim \f {1} {k!}$. Due to recent improvements concerning asymptotic independence on Wiener chaoses, \cite{Nourdin-Rosinski,Nourdin-Nualart-Peccati} we are also able to obtain results in a multivariate setting, see Theorem \ref{theorem-mixed} below as well as section \ref{section-ai}.

Depending on the chaos rank, $w$, of $G$, see definition \ref{def-chaos-rank} below, and $H$ we obtain a similar picture as in the Gau{\ss}ian case. In case $\f {(H-1)w} {m}+1 < \f 1 2$ the limiting process is given by a Wiener process, whereas for $\f {(H-1)w} {m}+1 > \f 1 2$ one again obtains  a Hermite process.

After proving the functional limits theorems we apply them  to the homogenization of fast/slow systems. We use the continuity of solutions to young and rough differential equation with respect to their drivers. As continuous maps preserve weak convergence this enables us to conclude weak convergence of the solutions to our ODE's by proving weak convergence of the drivers in H\"older/rough path topologies, see Theorem \ref{thm-application} below. For other applications of this method, c.f. \cite{roughflows,Kelly-Melbourne,Chevyrev-Friz-Korepanov-Melbourne-Zhang,Gehringer-Li-tagged}. 
\subsection{Statement of Results}

\begin{definition}\label{def-chaos-rank}
Given a stationary process $y_s$ that belongs to the $L^2$ space generated by a  Wiener process and a function $G$ such that $G(y_s) \in L^2(\Omega)$ we say that $G$ has chaos rank $w$ with respect to $y_s$ if and only if all projections of $G(y_s)$ onto the first $w-1$ Wiener chaoses are $0$ and  the projection of $G(y_s)$ onto the $w^{th}$ chaos is non-zero. 
\end{definition}
\begin{remark}
$G$ being centred with respect to the invariant distribution of $y_s$ is equivalent to $G$ having chaos rank bigger or equal $1$. In case $y_s$ is a normalized Gau{\ss}ian this coincides with the Hermite rank of $G$, see \cite{Taqqu-75}. 	
\end{remark}

\begin{convention}
If $G^j(X) = \sum_{k=0}^{\infty} c_{j,k} X^k$, where $j \in \{ 1,2,\dots, N \} $, are functions with chaos rank $w_j$ with respect to $y_t$, we order them  in such a manner that $H^*(w_j) < \f 1 2 $ for $j \leq n$ and $H^*(w_j) > \f 1 2$ in case $j> n$, for some $n \in \{ 0, \dots , N\}$. 
\end{convention}
For $T \in [0,1],$ set
	$$
	\bar{G}^{j,\epsilon}_T = \epsilon^{H^*(w_j) \vee \f 1 2} \int_0^{\f T \epsilon} G^j(y_t) dt.
	$$

The following is the main theorem of this article.
\begin{theorem*}\label{theorem-mixed}
	Fix $H \in ( \f 1 2, 1)$, $m \in \N$. Let $x$ be a kernel satisfying Assumption \ref{assumption-decay-correlation-kernel} and set $$y_t = \int_{-\infty}^t x(t-s) dZ^{H,m}_s.$$ For $j \in \{ 1, \dots, N \} $ let  $G^j(X) = \sum_{k=0}^{\infty} c_{j,k} X^k$ such that $G^j$ has chaos rank $w_j$ with respect to $y_t$ and $ \vert c_{j,k} \vert \lesssim \f 1 {k!}$ be given. 
	Assume further that $H^*(w_j) \not =\f 1 2 $.  Then, the following statements hold.
\begin{itemize}
\item 	
	As $\epsilon \to 0$, $( \bar{G}^{1,\epsilon}_T , \dots, \bar{G}^{N,\epsilon}_T)$
	converges  weakly  in $\C^{\gamma}\left( [0,1] , \R^N \right)$, for $\gamma \in (0, \f 1 2)$ in case $n>0$ and  $\gamma \in (0, \min_{j \in \{ 1, \dots , N \} } H^*(w^j)) $ otherwise.
\item	The limit  is a vector valued process $ \left(W_T,Z_T\right)$, where the first part $W_T$ is a $n$-dimensional Wiener process and
	$Z_T$ a $N-n$ dimensional Hermite process. 
	Furthermore,
	\begin{enumerate}
		\item [(1)]   $W_T \in \R^n$  and $Z_T \in \R^{N-n}$ are independent.
		\item [(2)]  The Wiener part $W_T$ of the process has the following covariance structure, 
\begin{align*}
\E\left[ W^j_{T} W^l_{S} \right]&= 2 ( T \wedge S) \int_0^{\infty} \E \left[ G^{j}(y_s) G^{l}(y_0) \right] ds\\
&=  2 (T \wedge S) \sum_{k,k'=0}^{\infty} \sum_{d=0}^{\infty} \int_0^{\infty} \E \left[ I_d(h^{d,k}_s) I_d(h^{d,k'}_0)  \right] ds
\end{align*}
		
\item [(3)] Each component of the  Hermite part of the process $Z_T$ has a representation by a Wiener process, which is the same for all component, that is  independent of $ W_T$:
$$Z_T = \left( \kappa_{n+1}  Z^{H^*(w_{n+1}),w_{n+1}}_T, \dots , \kappa_N  Z^{H^*(w_N),w_N}_T \right),
$$
where $\kappa_j = \lim_{\epsilon \to 0} \Vert  \bar{G}^{j,\epsilon}_1 \Vert_{L^2}$. 	\end{enumerate}
\end{itemize}\end{theorem*}
\begin{remark}
For $T,S \in [0,1]$, $Z$ thus has the following covariance structure: $$ \E\left[ Z^{H^*(w_{j}),w_{j}}_T Z^{H^*(w_{l}),w_{l}}_S \right] =  \delta^{j,l} \kappa_{j} \kappa_{l}   \int_{\R^{w_j}} d^{w_j} \xi \int_{0}^T \prod_{q=1}^{w_j} (r-\xi_q)^{ \f {H^*(w_j)-1} {w_j} -\f 1 2} dr  \int_{0}^S \prod_{q=1}^{w_l} (r'-\xi_q)^{ \f {H^*(w_j)-1} {w_j} -\f 1 2} dr'. 
$$ 
\end{remark}
\begin{remark} Given a function $G$,
it is not obvious how to construct a power series of a specified chaos rank. However,  we can center $G$ with respect to the first $w-1$ chaoses.  To do so set denote by   $P_{w-1}(G(y_s))$ the projection of  $G(y_s)$  onto the first $w-1$ Wiener chaoses  and and set
$\mathfrak{G}(y_s)= G(y_s) - P_{w-1}(G(y_s))$. Then,  $\mathfrak{G}(y_s)$ behaves like a function with chaos rank $w$ and Theorem \ref{theorem-mixed} is still applicable to $\mathfrak{G}$.
\end{remark}

As an application of the above result we  obtain the following theorem.
\begin{theorem*}\label{thm-application}
Let $y_t$ and $G$ be given as in Theorem A.

	Fix $ f \in \C^3_b([0,1],\R) $, $x_0 \in \R$, and consider the equations
	\begin{equation}\label{limit-eq}
	d x_t^\epsilon =\f{ \alpha(\epsilon)} {\epsilon} f(x_t^\epsilon) G(y^{\epsilon}_t) dt, 
	\qquad  x_0^\epsilon=x_0,
	\end{equation} 
where
	$$\alpha(\epsilon) = \begin{cases}
	\sqrt{\epsilon} &\hbox{ in case } H^*(w) < \f 1 2\\
	\epsilon^{H^*(w)} & \hbox{ in case } H^*(w) > \f 1 2,
	\end{cases} 
	$$
	and $y^{\epsilon}_t = y_{\f t \epsilon}$. Then, the following holds.
	\begin{enumerate}
		\item If
		$H^*(w) > \f 1 2$, $x_t^\epsilon$ converges weakly in $\C^{\gamma}([0,1],\R)$  to the solution to the Young differential equation
		$d\bar x_t = c f(\bar x_t) \,dZ_t^{H^*(w),w}$  with initial value $x_0$, for $\gamma \in (0, H^*(w))$ and 
		$$ c= \lim_{\epsilon \to 0 } \Vert \epsilon^{H^*(w)} \int_0^{ \f 1 \epsilon} G(y_s) ds \Vert_{L^2(\Omega)}.$$
		\item If
		$H^*(w) < \f 1 2$, 
		$x_t^\epsilon$ converges weakly  in  $\C^\gamma([0,1],\R)$ to the solution of the Stratonovich stochastic differential equation
		$d\bar x_t = c  f(\bar x_t) \circ \,dW_t$ with $ \bar x_0=x_0$, for $\gamma \in(0, \f 1 2)$, $W$ denotes a standard Wiener process and 
		$$ c= \lim_{\epsilon \to 0 } \Vert \sqrt{\epsilon} \int_0^{ \f 1 \epsilon} G(y_s) ds \Vert_{L^2(\Omega)}.$$
	\end{enumerate} 
\end{theorem*}

\section{Notation}
\begin{enumerate}
	\item $(\Omega,\F,\P)$ denotes our underlying probability space 
	\item $\lambda$ denotes the Lebesgue measure on the respective spaces or a parameter of the Hermite Ornstein-Uhlenbeck process
	\item $B$ denotes a two sided standard Brownian motion
	\item $\hat{B}$ denotes a Gau{\ss}ian complex-valued random spectral measure to be defined in section \ref{subsubsection-spectral-measure}
	\item $I_d$ denotes a $d$ dimensional Wiener isometry given by iterated Wiener integrals
	\item $\hat{I}_d$ denotes a $d$ dimensional iterated Wiener integral with respect to $\hat{B}$.
	\item $H$ denotes the self-similarity of our underlying Hermite process
	\item $m$ denotes the rank of our underlying Hermite process
	\item $H_0= 1+ \f {H-1} {m}$
	\item $H^*(d)=(H_0-1)d+1$
	\item $f(x) \lesssim g(x)$ denotes less or equal up to a constant; there exists a constant $M$ such that for all $x$, $f(x) \leq M g(x)$
	\item $f_+ = \max(0,f)$ 
	\item As we have to deal with many integrals we sometimes use the notation $\int dx f(x)$ instead of $\int f(x) dx$. Furthermore, instead of $ \int_{\R^k} ds_1 \dots ds_k f(s_1, \dots, s_k)$ we sometimes write $ \int_{\R^k} d^ks f(s_1, \dots, s_k)$. Here $d^k s$ denotes that we integrate over the $s$ variables of which there are exactly $k$. Sometimes the index is not $1$ to $k$, due to possible double indices, but, nevertheless, $d^ks$ denotes that there are exactly $k$ of them.
\end{enumerate}

\section{Preliminaries}\label{section-preliminary}

Given $f \in L^2(\R^a,\lambda)$ we denote its $a^{th}$ multiple Wiener-It\^o integral by $I_a(f)= \int_{\R^a} f(\xi_1, \dots ,\xi_a) dB_{\xi_1} \dots dB_{\xi_a}$, where $B$ is a two-sided Wiener process and the integral does not run over diagonals.
For symmetric functions $f$ we obtain $I_a(f) = a! \int_{\R} \int_{-\infty}^{\xi_1} \dots \int_{-\infty}^{\xi_{a-1}} f(\xi_1, \dots , \xi_{a}) dB_{\xi_1} \dots dB_{\xi_a}$.
Furthermore, let $\tilde{f}$ denote $f$'s symmetrization, then,  $I_a(f) = I_a(\tilde{f})$.
In particular, for $f \in L^2(\R^a,\lambda)$ and $g \in L^2(\R^b,\lambda)$,
\begin{equation*}\label{eqaution-L2-isometrie}
\E \left[ I_a(f) I_b(g) \right] =  \delta_{a,b} a! \< \tilde{f}, \tilde{g} \>_{L^2(\R^a,\lambda)}.
\end{equation*}
For $ r \leq a \wedge b $ we denote the $r^{th}$ contraction between $f$ and $g$ by 
$$f \otimes_r g (\xi_1, \dots, \xi_{a+b-2r}) = \int_{\R^r} f(\xi_1, \dots, \xi_{a-r}, s_1, \dots , s_r)   g(\xi_{a-r+1}, \dots, \xi_{a+b-2r}, s_1, \dots , s_r) ds_1 \dots ds_r.$$
Moreover, we denote its symmetrization by $f \tilde{\otimes}_r g$, see  \cite{Nualart}. We conclude this section with an identity which we use plentifully in section \ref{section-decomposition}.
\begin{lemma}\label{product-formulae}
Given symmetric functions $f \in L^2(\R^a,\lambda)$ and $g \in L^2(\R^b,\lambda)$, then,  the following relation holds
\begin{equation}\label{equation-produktformel}
I_a(f) I_b(g) = \sum_{r=0}^{a \wedge b} r! {a \choose r} {b \choose r} I_{a+b-2r} ( f \tilde{\otimes}_{r} g ).
\end{equation}
\end{lemma}
\subsection{Hermite processes}

\begin{definition}\label{Hermite-processes}
Let $m\in \N$, $H \in (\f 1 2,1)$ be given and recall, $H_0= 1+ \f {H-1} {m}$. The class of Hermite processes of rank $m$ is given by the following mean-zero processes,
\begin{equation}\label{definition-Hermite-processes}
Z_t^{H,m}= K(H,m) \int_{\R^m} \int_0^t \prod_{j=1}^m (s-\xi_j)_+^{  H_0 - \f 3 2 } \, ds \,  d  B_{\xi_1} \dots d  B_{\xi_m},
\end{equation}
where the integral over $\R^m$ is to be understood as a multiple Wiener-It\^o integrals, meaning no integration along the diagonals, and the constant $K(H,m)$ is chosen so that their variances are $1$ at $t=1$. The number $H$ is also called Hurst parameter and 
$K(H,m)^2 = \f {H(2H-1) } { B(H_0 - \f 1 2 , 2 -2H_0)},$
where $B$ denotes the beta function, c.f. \cite{Maejima-Ciprian,Taqqu}.
\end{definition}
Hermite processes have stationary increments, are self-similar with exponent $H$,
$$ \lambda^H  Z^{H,m} _{\f \cdot\lambda} \sim  Z^{H,m}_{\cdot} ,$$	
and their covariance is given by 
\begin{equation}
\E[ Z_t^{H,m} Z_s^{H,m}] =  \f 1 2 \left(   t^{2H} + s^{2H} - \vert t-s \vert^{2H}\right).
\end{equation}
They live in the $m^{th}$ Wiener chaos and thus, by hypercontractivity,  have finite moments of all orders. Using Kolmogorv's theorem, one can show that the Hermite processes $Z_t^{H,m}$ have sample paths of H\"older regularity up to  $H$. The rank $1$  Hermite processes $Z^{H, 1}$ are fractional Brownian motions for $H> \f 1 2$, as above definition matches the 
Mandelbrot Van-Ness representation, see (\ref{definition-Hermite-processes}) and \cite{Pipiras-Taqqu}.

In \cite{Taqqu} the following spectral representation for Hermite processes was obtained using a Gau{\ss}ian random spectral measure $\hat{B}_{\hat{\xi}}$, see section \ref{subsubsection-spectral-measure} for a brief summary, 
\begin{equation*}\label{spectral-representation-Hermiteprocesses}
Z^{H,m}_t = \hat K(H,m) \int_{\R^m} \f { e^{i t \sum_{j=1}^m \hat{\xi}_j} -1 } { i \sum_{j=1}^m \hat{\xi}_j} \prod_{j=1}^m \vert \hat{\xi}_j \vert^{H_0 - \f 1 2}  d \hat{B}_{\hat{\xi}_1} \dots d \hat{B}_{\hat{\xi}_m},
\end{equation*}
where $\hat{K}(H,m)$ is chosen such that $ \E \left[ \left(Z^{H,m}_1 \right)^2 \right]=1.
$

\subsubsection{Wiener integrals for Hermite processes }
In \cite{Maejima-Ciprian} Wiener integrals with respect to Hermite processes were introduced. Via an isometry construction it was shown that for $f \in \mathcal{H}$, where $ \mathcal{H} =\{ f : \R \to \R \, :  \int_{\R^2}  f(u) f(v)  \vert u-v\vert^{2H-2} du dv < \infty \}$, the usual Wiener integral approach makes sense. In order to avoid integrability problems, we impose a bit more regularity and restrict ourselves to the space $ \vert \mathcal{H} \vert =\{ f : \R \to \R \, :  \int_{\R^2} \vert f(u) \vert \vert f(v) \vert \vert u-v\vert^{2H-2} du dv < \infty \}$. Note that $L^1(\R,\lambda) \cap L^2(\R,\lambda) \subset \vert \mathcal{H} \vert $, see \cite{Maejima-Ciprian, Pipiras-Taqqu}. On $\mathcal{H}$ the following relation holds,
\begin{align}\label{equation-Hermite-isometry}
\int_{\R} f(s) dZ^{H,m}_s &= K(H,m) \int_{\R^m} \int_{\R} f(s) \prod_{j=1}^m (s- \xi_{j})^{H_0- \f 3 2} ds dB_{\xi_{1}} \dots dB_{\xi_{m}}\\
&= K(H,m) I_m\left( \int_{\R} f(s) \prod_{j=1}^m (s- \xi_{j})^{H_0- \f 3 2} ds\right),
\end{align}
where the integral over $\R^m$ is to be understood as an iterated Wiener integral. By the above and using the identity
\begin{equation}\label{identity-beta-function}
\int_{\R} (u-y)_{+}^{a-1} (v-y)_{+}^{a-1} dy = B(a,2a-1) \vert u-v\vert^{2a-1},
\end{equation}
where $B$ denotes the beta function, one obtains the following relation,
\begin{equation}\label{equaion-L2-isometry}
\E \left[ \int_{\R} f dZ^{H,m} \int_{\R} g dZ^{H,m} \right] = H(2H-1) \int_{\R^2}  f(u)  g(v)  \vert u-v\vert^{2H-2} du dv.
\end{equation}
Furthermore, we denote 
\begin{align*}
\Vert f \Vert_{\mathcal{H}} &= H(2H-1) \int_{\R^2} f(u) f(v) \vert u-v\vert^{2H-2} du dv\\
\Vert f \Vert_{\vert \mathcal{H} \vert } &= H(2H-1) \int_{\R^2} \vert f(u) \vert \vert f(v) \vert \vert u-v\vert^{2H-2} du dv.
\end{align*} 
\subsubsection{Hermite Ornstein-Uhlenbeck processes}
Using this integration theory, the Hermite Ornstein-Uhlenbeck process was introduced in \cite{Maejima-Ciprian}. It is the unique solution to the following SDE, where $\lambda, \sigma >0$,

\begin{equation}\label{definition-HermiteOU}
y_t = y_0 - \lambda \int_0^t y_s ds + \sigma Z^{H,m}_t,
\end{equation}
which is given by
\begin{equation}\label{equation-solution-HermiteOU}
y_t = e^{-\lambda t} \left( y_0 + \sigma \int_0^t e^{\lambda s} d Z^{H,m}_s \right).
\end{equation}
By choosing $y_0 \sim \sigma \int_{-\infty}^0  e^{\lambda s} d Z^{H,m}_s$ one obtains its stationary solution. Moreover, by the above formulae, solutions started with different initial conditions converge exponentially fast towards this stationary solution.

In \cite{Cheridito-Kawaguchi-Maejima} a formulae for the covariance decay in case $m=1$ was obtained. As the covariance does not change in $m$ and $ e^{-(t-s)} \1_{s \leq t} \in L^1(\R) \cap L^2(\R)$ the same formulae also holds true for the Hermite Ornstein-Uhlenbeck processes and is given by

\begin{equation}
\E \left[ y_t y_{t+s}\right] = \f 1 2 \sigma^2 \sum_{n=1}^N \lambda^{-2n} \left( \prod_{j=0}^{2n-1} (2H-j) \right) s^{2H-2n} + O(s^{2H-2N-2}),
\end{equation}
for $s \to \infty$, see \cite{Maejima-Ciprian}. In particular,
\begin{equation}\label{equation-HermiteOU-corr-decay}
\vert \E[ y_t y_{t+s}] \vert \lesssim 1 \wedge s^{2H-2}.
\end{equation}

\subsubsection{Volterra  processes}
For $x \in \mathcal{H}$ we may also define 

\begin{align}
y_t &= \int_0^t x(t-s) dZ^{H,m}_s\\
&= K(H,m) I_m\left(\int_{0}^t x(t-s) \prod_{j=1}^m (s- \xi_{j})^{H_0- \f 3 2} ds\right).
\end{align}
By setting $x(s)= e^{-s}$ we obtain the Hermite Ornstein-Uhlenbeck process with inital value $0$. Furthermore, setting
\begin{align}\label{definition-Volterra-process}
y_t&= \int_{\infty}^t x(t-s) dZ^{H,m}_s\\
&= K(H,m) I_m\left(\int_{-\infty}^t x(t-s) \prod_{j=1}^m (s- \xi_{j})_{+}^{H_0- \f 3 2} ds\right)
\end{align} 
leads to a class of stationary processes, which in case $x=e^{-s}$ equals the stationary Hermite Ornstein-Uhlenbeck process.
In our analysis below we treat processes given as such Volterra integrals and, as in \cite{Diu-Tran}, impose the condition $x \in L^1(\R)$. 
We restrict our analysis to the stationary case, $y_t = \int_{-\infty}^t x(t-s) dZ^{H,m}_s$, though many estimates still hold true without this assumption. Furthermore, we impose the following decay condition on the kernel $x$. This ensures a similar decay of covariances as in the Hermite Ornstein-Uhlenbeck case, see equation \ref{equation-HermiteOU-corr-decay}.
\begin{assumption}\label{assumption-decay-correlation-kernel}
Assume $x \in L^1(\R) \cap \vert \mathcal{H} \vert $ and 
\begin{equation*}
\int_{\R^2} \vert x(t-u) x(t'-v) \vert \vert u-v \vert^{2H-2} du dv  \lesssim 1 \wedge \vert t-t' \vert^{2H-2}.
\end{equation*}
In particular, for $y_t$ defined in equation (\ref{definition-Volterra-process})  this leads to ,
\begin{equation*}
\vert \E [   y_t y_{t'}   ] \vert  \lesssim 1 \wedge \vert t-t' \vert^{2H-2}.
\end{equation*}
\end{assumption}

\begin{remark}\label{remark-assumption-covariance}
In \cite{Nourdin-Nualart-Zintout} amongst other things the case $m=1$ in the short range dependent setting was treated. One of their assumption on the kernel $x$ is the following integrability condition,
$$
\int_{\R} \left( \int_{[0,\infty]^2} x(u) x(v) \vert u-v-a \vert du dv \right)^m da < \infty,
$$
where $m$ denotes the Hermite rank of the function $G$. In case we are as well in the short range dependent regime, using the  notion of the chaos rank of $G$ instead of the Hermite rank,
Assumption \ref{assumption-decay-correlation-kernel} implies this condition. 
\end{remark}

\section{Decomposition and convergence  for  building blocks}
In this section we first decompose each polynomial $(y_t)^k$ into its Wiener chaos components. To do so we apply the product formulae, Lemma \ref{product-formulae}, iteratively and collect all obtained "building blocks"  belonging to the same Wiener chaos. This terms are then at the centre of our investigation as we can obtain the general case by sums of these objects. Next, we analyse the variance growth of these terms to obtain our scaling rate. Finally, we prove convergence in finite dimensional distributions for the rescaled integrals of our building blocks. Here we distinguish, as in the Gau{ss}ian setting, the short and long range dependent regime. In the first one our limit is given by a Wiener process and we make use of the Fourth Moment Theorem to conclude our result. In the latter, the limits are again Hermite processes and, as in \cite{Diu-Tran}, we argue via first rewriting everything as a multiple Wiener It\^o integral with respect to a Gau{\ss}ian complex-valued random spectral measure and then prove $L^2$ kernel convergence.

\begin{remark}
Henceforth we suppress the constants $K(H,m)$ in our notation.
\end{remark}

\subsection{Decomposition}\label{section-decomposition}
Given $y_t= I_m\left(\int_{-\infty}^t x(t-s) \prod_{j=1}^m (s- \xi_{j})_{+}^{H_0- \f 3 2} ds\right)$, we aim to calculate the contribution of $(y_t)^k$ to each distinct Wiener chaos. As this kernel appears again and again we set $f_t =  \int_{-\infty}^t x(t-s) \prod_{j=1}^m (s- \xi_{j})_{+}^{H_0- \f 3 2} ds$. Now, by iteratively applying Lemma \ref{product-formulae}, we obtain,
\begin{align}\label{equation-polynomial-decomposition}
(y_t)^k &= \left( I_m\left(\int_{-\infty}^t x(t-s) \prod_{j=1}^m (s- \xi_{j})_{+}^{H_0- \f 3 2} ds\right) \right)^k\\
&= \left( I_m\left(\int_{-\infty}^t x(t-s) \prod_{j=1}^m (s- \xi_{j})_{+}^{H_0- \f 3 2} ds\right) \right)^{k-2} \left( \sum_{r_1=0}^{m} (r_1)! {m \choose r_1 } {m \choose r_1} I_{2m-2r_1} ( f_t \tilde{\otimes}_{r_1} f_t )  \right)\\
&=  \left( I_m\left(\int_{-\infty}^t x(t-s) \prod_{j=1}^m (s- \xi_{j})_{+}^{H_0- \f 3 2} ds\right) \right)^{k-3}\\
&\left( \sum_{r_1=0}^m (r_1)! {m \choose r_1 } {m \choose r_1 } \sum_{r_2=0}^{\left( 2m-2r_1 \right)\wedge m}  (r_2)! {m \choose r_2 } {2m-2r_1 \choose r_2}  I_{3m-2r_1-2r_2} ( f_t \tilde{\otimes}_{r_1} f_t \tilde{\otimes}_{r_2} f_t ) \right)\\
&= \sum_{r_1, r_2, \dots, r_k} C_1(r_1, \dots, r_k,k,m) I_{km-2\sum_{j=1}^k r_j} ( f_t \tilde{\otimes}_{r_1} f_t  \tilde{\otimes}_{r_2}  \dots \tilde{\otimes}_{r_k} f_t ),
\end{align}
where  $r_1 \leq m $ ,  $r_2 \leq (2m-2r_1) \wedge m$ , $ \dots$, $ r_k \leq \left( km-2\sum_{j=1}^k r_j \right) \wedge m $ and 
$$ 
C_1(r_1, \dots, r_k,k,m) = \prod_{j=1}^k (r_j)! { m \choose r_j} { jm- 2 \sum_{l=1}^{j-1} r_l \choose r_j}
$$ denotes the arising constants.
Henceforward, we denote the tuple  $(r_1, \dots, r_k)$ just by $r$ and set $\delta(k,r)=km-2\sum_{j=1}^k r_j$.
The above calculation shows that we can decompose polynomials into their distinct Wiener chaos parts once we know how to compute the terms $f_t \tilde{\otimes}_{r_1} f_t  \tilde{\otimes}_{r_2}  \dots \tilde{\otimes}_{r_k} f_t$.

Therefore, we now  investigate contractions of $f_t$ with itself a bit more.
In \cite{Diu-Tran} the following was shown,
\begin{align*}
&f_t \otimes_{r_1} f_t(\xi_{1,1}, \dots, \xi_{1,m-r_1},\xi_{2,1}, \dots, \xi_{2,m-r_1}) \\
&= \int_{[-\infty,t]^2}  ds_1  ds_2 x(t-s_1) x(t-s_2)  \prod_{l=1}^{m-r_1} (s_1- \xi_{1,l})_{+}^{H_0- \f 3 2}  (s_2- \xi_{2,l})_{+}^{H_0- \f 3 2}  \int_{\R^{r_1}} d^{r_1} z \prod_{j=1}^{r_1} (s_1-z_j)_{+}^{H_0 - \f 3 2} (s_2-z_j)_{+}^{H_0 - \f 3 2}  \\
&= \int_{[-\infty,t]^2}  ds_1 ds_2 x(t-s_1) x(t-s_2) \prod_{l=1}^{m-r_1} (s_1- \xi_{1,l})_{+}^{H_0- \f 3 2}  (s_2- \xi_{2,l})_{+}^{H_0- \f 3 2}   \left( \int_{\R} dz  (s_1-z)_{+}^{H_0 - \f 3 2} (s_2-z)_{+}^{H_0 - \f 3 2} \right)^{r_1}\\
&=  C_2(H_0)^{r_1} \int_{[-\infty,t]^2} ds_1 ds_2 x(t-s_1) x(t-s_2) \prod_{l=1}^{m-r_1} (s_1- \xi_{1,l})_{+}^{H_0- \f 3 2}  (s_2- \xi_{2,l})_{+}^{H_0- \f 3 2}   \vert s_1-s_2 \vert^{r_1 (2 H_0 -2 )} ,
\end{align*}
where $C_2(H_0) = B(H_0 - \f 1 2, 2- 2H_0)$ and $B$ denotes the beta function, see also Equation \ref{identity-beta-function}. Hence,
\begin{align*}
f_t \tilde{\otimes}_{r_1} f_t &= \f { C_2(H_0)^{r_1}} {(2m-2r_1)!} \sum_{\psi_1 \in \mathcal{S}_{2m-2r_1} } \int_{[-\infty,t]^2}  ds_1 ds_2 x(t-s_1) x(t-s_2) \\
&\prod_{l=1}^{m-r_1} (s_1-\xi_{\psi_1(1,l)})^{H_0 - \f 3 2}_{+} (s_2-\xi_{\psi_1(2,l)})^{H_0 - \f 3 2} _{+}
\vert s_1-s_2 \vert^{r_1 (2- 2H_0 )},
\end{align*}
where $\mathcal{S}_{2m-2r_1}$ denotes the symmetric group of order $2m-2r_1$ and we implicitly make the identifications  $(1,l)=l$ and $(2,l)=m-r_1+l$. From now on we freely use such implicit identifications of indices to lighten the notation. 

When computing the next contraction, $f_t \tilde{\otimes}_{r_1} f_t \otimes_{r_2} f_t$, we face the problem that for different choices of $\psi_1 \in S_{2m-2r_1}$ we eventually integrate different terms. This is due to the fact that contractions use the "last" $r_2$ variables, which is well defined for the symmetric function $f_t \tilde{\otimes}_{r_1} f_t$, however, in each of the summands the notion of the "last" variables depends on the permutation. Nevertheless, we know that exactly $r_2$ $\xi's$ are consumed in the next round and we denote by $r_{2,1}(\psi_1)$ and $r_{2,2}(\psi_1)$ the amount which contracts with the $s_1$ and the $s_2$ terms respectively.
Performing the same calculation as above, we obtain, 
\begin{align*}
&f_t \tilde{\otimes}_{r_1} f_t \otimes_{r_2} f_t =  \f { C_2(H_0)^{r_1}    C_2(H_0)^{r_2} } {(2m-2r_1)!} \sum_{\psi_1 \in \mathcal{S}_{2m-2r_1} } \int_{[-\infty,t]^3}  ds_1 ds_2 ds_3 x(t-s_1) x(t-s_2) x(t-s_3) \\
&\prod_{l=1}^{m-r_1-r_{2,1}(\psi_1)} (s_1-\xi_{\psi_1(1,l)})^{H_0 - \f 3 2}_{+} 
\prod_{l=1}^{m-r_1-r_{2,2}(\psi)} (s_2-\xi_{\psi_1(2,l)})^{H_0 - \f 3 2}_{+}\\
&\prod_{l=1}^{m-r_2} (s_3-\xi_{3,l})^{H_0 - \f 3 2}_{+} \vert s_1-s_2 \vert^{r_1 (2H_0 -2 )} \vert s_1- s_3 \vert^{r_{2,1}(\psi_1)(2H_0 -2 )} \vert s_2-s_3 \vert^{r_{2,2}(\psi_1)(2H_0 -2 )}, 
\end{align*}
together with the algebraic constraint $r_{2,1}(\psi_1) + r_{2,2}(\psi_1) = r_2$. After forming the symmetrization of this expression one ends up with,
\begin{align*}
&f_t \tilde{\otimes}_{r_1} f_t \tilde{\otimes}_{r_2} f_t\\
&= \f { C_2(H_0)^{r_1 } } {(2m-2r_1)!} \f {  C_2(H_0)^{r_2  }} {(3m-2r_1-2r_2)! } \sum_{\psi_1 \in \mathcal{S}_{2m-2r_1} } \sum_{\psi_2 \in \mathcal{S}_{3m - 2 r_1 - 2r_2}}  \\ &\int_{[-\infty,t]^3}  ds_1 ds_2 ds_3 x(t-s_1) x(t-s_2) x(t-s_3) \\
&\prod_{l=1}^{m-r_1-r_{2,1}(\psi_1)} (s_1-\xi_{\psi_2(\psi_1(1,l))})^{H_0 - \f 3 2}_{+}
\prod_{l=1}^{m-r_1- r_{2,2}(\psi)} (s_2-\xi_{\psi_2(\psi_1(2,l))})^{H_0 - \f 3 2}_{+} \prod_{l=1}^{m-r_2} (s_3-\xi_{\psi_2(3,l)})^{(2H_0 -2 )}_{+} \\
& \vert s_1-s_2 \vert^{r_1 (2H_0 -2 )} 
\vert s_1- s_3 \vert^{r_{2,1}(\psi_1)(2H_0 -2 )} \vert s_2-s_3 \vert^{r_{2,2}(\psi_1)(2H_0 -2 )} 
.
\end{align*}

Thus, each time we contract once more, we integrate over one more variable resulting in the gain of an additional kernel, giving rise to the term $x(t-s_1) x(t-s_2) x(t-s_3)$, which does not further interact with anything else. However, the terms $ \prod (s_j - \xi_{\psi_1(j,l)})^{H_0 - \f 3 2}_{+}$ lead to an entanglement of the old variables with the new one, in case the contraction number is greater than $0$. Although this entangling depends on the choice of the permutation, the obtained structure is the same.

If  we perform $k$ contractions of $f_t$ with itself keeping track of all the constants $C_2(H_0)^{\1_{r_1>0}  }$, renormalization factors $ \f {1} {(2m-2r_1)!}$, sums $\sum_{\psi_1 \in \mathcal{S}_{2m-2r_1} }$ as well as the range of the products  is  notationally quite intense. Hence, we assume from now on  that we performed $k$ contractions ,with contraction numbers $r=(r_1,r_2, \dots, r_k)$, giving rise to one object that represents all possible choices,   $\psi \in \mathcal{S}_r =  \mathcal{S}_{2m-2r_1} \times \mathcal{S}_{3m-2r_1-2r_2} \dots  \times \mathcal{S}_{km -2 \sum_{j=1}^k r_j}$ and denote the arising constants by $C_3(\psi,r,H_0)$. Furthermore, we introduce constants $M_j(\psi)$  and $\beta_{j,q}(\psi)$ to keep track of the range of the products as well as the exponent of $\vert s_j-s_q\vert$ respectively.
The expression we then end up with looks more structured and we summarize the conclusions of the above discussion in the following Lemma.
\begin{lemma}\label{lemma-contractions-calculated}
Given a kernel $x$ such that, $f_t =  \int_{-\infty}^t x(t-s) \prod_{j=1}^m (s- \xi_{j})_{+}^{H_0- \f 3 2} ds$ belongs to $L^2\left(\R^m,\lambda\right)$ and contraction numbers $r=(r_1, \dots, r_k),$ where  $r_1 \leq m $ ,  $r_2 \leq \left(2m-2r_1\right) \wedge m$ , $ \dots$, $ r_k \leq \left( km-2\sum_{j=1}^k r_j \right) \wedge m $, then, the following identity holds,
\begin{align}
&f_t \tilde{\otimes}_{r_1} f_t \tilde{\otimes}_{r_2} f_t \dots \tilde{\otimes}_{r_k} f_t\\ 
&=   \sum_{\psi \in \mathcal{S}_r} C_3(\psi,r,H_0) \int_{[-\infty,t]^k}  d^k s \prod_{j=1}^k x(t-s_j) 
\prod_{l=1}^{M_j(\psi)} (s_j-\xi_{\psi(j,l)})^{H_0 - \f 3 2} 
\prod_{q=1}^{j-1 }\vert s_j - s_q \vert^{\beta_{j,q}(\psi) (2H_0 -2 )}, \\
\end{align}
together with the algebraic constraints,
\begin{align*}
\sum_{q=1}^{j-1} \beta_{j,q}(\psi) = r_j, \, \, \, \, \, \, \, \, \, \, \, \, 
\sum_{j=1}^k M_j(\psi) = km - 2 \sum_{j=1}^k r_j = \delta(k,r),
\end{align*}
independently of the choice of $\psi \in \mathcal{S}_r$.
\end{lemma}
Hence, to decompose $(y_t)^k$ into its distinct Wiener chaos parts it is only left to collect all terms $f_t \tilde{\otimes}_{r_1} f_t \tilde{\otimes}_{r_2} f_t \dots \tilde{\otimes}_{r_k} f_t$ which give the same value for $\delta(k,r)= km - 2 \sum_{j=1}^k r_j $.
To simplify our notation we set $g^{k,r}_t = f_t \tilde{\otimes}_{r_1} f_t \tilde{\otimes}_{r_2} f_t \dots \tilde{\otimes}_{r_k} f_t$ and 
$h^{d,k}_t = \sum_{ \{r : \delta(k,r) =d\} } C_1(r,k,m)g^{k,r}_t$.
Now, the terms $h^{d,k}_t$ equal the projection of $(y_t)^k$ onto the $d^{th}$ Wiener chaos, 
thus, by Equation \ref{equation-polynomial-decomposition},
\begin{align*}
(y_t)^k &= \sum_{d=0}^{km} I_d \left( \sum_{ \{ r:  \delta(k,r) =d \} } C_1(r,k,m) g^{k,r}_t \right)\\
&= \sum_{d=0}^{km} I_d(h^{d,k}_t) 
\end{align*}

\subsection{Examining the scaling behaviour}\label{section-variance-computation}

In this section we bound the variance growth of terms of the form $ \int_0^{\f T \epsilon} h^{d,k}_t dt$ to obtain our scaling rate. Typically, one expects that this growth  decreases with $d$ and after some critical value stays at the Wiener scaling as long as no cancellation occurs, see  also Remark \ref{remark-assumption-covariance}and \cite{Taqqu,Taqqu_sums,Gehringer-Li-homo}. This change of behaviour appears due to the algebraic decay of the correlations proved in this section. Terms in the $d^{th}$ chaos admit a decay of the form $ 1 \wedge s^{d (2H_0-2)}$. Thus, $\int_0^{\f 1 \epsilon} 1 \wedge s^{d (2H_0-2)} ds$  either converges as $\epsilon \to 0$, if $d$ is large enough, or diverges at rate $\epsilon^{-d(2H_0-2)+1}$ leading to the change of behaviour.

\begin{lemma}\label{wachstum-konstanten}
Given $h^{d,k}_t$ as defined above  for a kernel $x$ satisfying Assumption \ref{assumption-decay-correlation-kernel}, then,  the following estimate holds,
$$
\Big \vert \E \left[ I_d (h^{d,k}_t) I_{d'}( h^{d',k'}_{t'}) \right] \Big \vert  \lesssim  C_4(k,k',m,d) \left( 1 \wedge \vert t- t'\vert^{(2H_0-2)d} \right),
$$
where $C_4(k,k',m,d) = \f {\sqrt{(km)!(k'm)!} (m!)^{(k+k')} (k+k')^m ((k+k')m)^2 m^{k+k'} \mathfrak{L}^{(k+k')m}} {d^2}$ and $\mathfrak{L} = C_2(H_0) +3 + \Vert x \Vert_{\mathcal{H}} + K(H,m)$.
\end{lemma}
\begin{proof}
In case $d \not = d'$ the expectation is $0$ by orthogonality of different Wiener chaoses. Thus, henceforward we assume $d=d'$.
Due to the decomposition,
\begin{align*}
\E \left[ I_d \left(h^{d,k}_t \right) I_d\left( h^{d,k'}_{t'}\right) \right] &= \sum_{ \{r : \delta(k,r) =d\} } C_1(r,k,m) \sum_{ \{r' : \delta(k',r') =d\} } C_1(r',k',m) \E \left[ I_d\left(g^{k,r}_t \right) I_d\left(g^{k',r'}_{t'}\right)  \right] 
\end{align*}
and the finiteness of both sums, we may restrict ourselves to the analysis of the behaviour of $\E \left[ I_d\left(g^{k,r}_t \right) I_d\left(g^{k',r'}_{t'}\right)  \right] $ as long as our bounds only depend on $d,k,k',m$ and $H_0$. By the Wiener-It\^o isometry $\E \left[ I_d\left(g^{k,r}_t \right) I_d\left(g^{k',r'}_{t'}\right)  \right] = d! \< g^{k,r}_t , g^{k',r'}_{t'} \>_{L^2(\R^d)}$ which is equivalent to computing the $d^{th}$ contraction between $g^{k,r}_t$ and $ g^{k',r'}_{t'}$. As both terms arise from iterated contractions of $f$ we may argue similarly as above to find a good expression for their $L^2$ norm. We treat the behaviour of multiplicate constants separately, thus, we suppress them in the following calculations. We have,
\begin{align*}
\E \left[ I_d(g^{k,r}_t ) I_d(g^{k',r'}_{t'})  \right] &= \int_{\R^d} d^d\xi   \sum_{\psi \in \mathcal{S}_r}  \int_{[-\infty,t]^k}  d^k s \prod_{j=1}^k x(t-s_j) 
\prod_{l=1}^{M_j(\psi)} (s_j-\xi_{\psi(j,l)})^{H_0 - \f 3 2} 
\prod_{q=1}^{j-1 }\vert s_j - s_q \vert^{\beta_{j,q}(\psi) (2H_0 -2 )}\\
&\sum_{\psi' \in \mathcal{S}_{r'}}  \int_{[-\infty,t']^{k'}}  d^{k'} s' \prod_{j'=1}^{k'} x(t'-s'_{j'}) 
\prod_{l'=1}^{M_{j'}(\psi')} (s'_{j'}-\xi_{\psi'(j',l')})^{H_0 - \f 3 2} 
\prod_{q'=1}^{j'-1 }\vert s'_{j'} - s'_{q'} \vert^{\beta'_{j',q'}(\psi') (2H_0 -2 )}.  \\
\end{align*}

When we  computed  $g^{k,r}_t = f_t \tilde{\otimes}_{r_1} f_t \tilde{\otimes}_{r_2} f_t \dots \tilde{\otimes}_{r_k} f_t$ we saw that, depending on the choice of permutations $\psi$, two variables $s_{j} $ and $s_{q}$ either interact with each other or not  leading to  the terms $\vert s_j -s_q \vert^{\beta_{j,q}(2H_0-2)}$, importantly, their exponents need to satisfy certain algebraic constraints. We may view each of the $k$ variables $s_j , j=1, \dots k$ as the node of a graph with $m$ "degrees of freedoms". Entangling two variables, $s_j$ and $s_q$, can be seen as taking away this freedoms and connecting them with $\beta_{j,q}$ edges. Once a node has $m$ connections it can no longer interact with any other variables in following contractions. 

However, in case $\delta(k,r)=d$ there are in total $d$ freedoms left, no matter which permutation we consider. When computing the variance of this objects we viewed this as having two, possibly distinct, graphs,  with $k$ and $k'$ nodes respectively such that each graph has $d$ "degrees of freedom" left. What is now left to do, is  to literally connect the dots to obtain the formulae below by the same calculations as in section \ref{section-decomposition}. This graph analogy will be used over and over and provides a good picture of what is going on. The exponents $\gamma_{j,j'}$ now depend on both permutations and denote how many edges are connecting the node representing $s_j$ to the one representing $s'_{j'}$. We again obtain constants of the form $C_2(H_0)$ which we, together with the constants previously obtained,  suppress in the following computations. We end up with,

\begin{align*}
\E \left[ I_d\left(g^{k,r}_t \right) I_d\left(g^{k',r'}_{t'}\right)  \right] &=  \sum_{\psi \in \mathcal{S}_r, \psi' \in \mathcal{S}_{r'}} \int_{[-\infty,t']^{k'}} d^{k'}s' \int_{[-\infty,t]^{k}} d^k s \prod_{j=1}^k \prod_{ j'=1}^{k'} x(t-s_j) x(t'-s_{j'}')  \\
& \vert s_j - s_{j'}'\vert^{\gamma_{j,j'}(\psi,\psi') (2H_0 -2)} \prod_{q=1}^{j-1} \vert s_j - s_q \vert^{\beta_{j,q}(\psi) (2H_0 -2)} \prod_{q'=1}^{j'-1}  \vert s_{j'}' - s_{q'}' \vert^{\beta_{j',q'}'(\psi') (2H_0 -2)}, 
\end{align*}
and, by the triangle inequality,
\begin{align*}
\Big \vert \E \left[ I_d\left(g^{k,r}_t \right) I_d\left(g^{k',r'}_{t'}\right)  \right] \Big \vert &\leq  \sum_{\psi \in \mathcal{S}_r, \psi' \in \mathcal{S}_{r'}} \int_{[-\infty,t']^{k'}} d^{k'} s' \int_{[-\infty,t]^{k}} d^k s \prod_{j=1}^k \prod_{ j'=1}^{k'}\vert  x(t-s_j) x(t'-s_{j'}') \vert \\
&  \vert s_j - s_{j'}'\vert^{\gamma_{j,j'}(\psi,\psi') (2H_0 -2)} \prod_{q=1}^{j-1} \vert s_j - s_q \vert^{\beta_{j,q}(\psi) (2H_0 -2)} \prod_{q'=1}^{j'-1}  \vert s_{j'}' - s_{q'}' \vert^{\beta_{j',q'}'(\psi') (2H_0 -2)}.
\end{align*}
From now on we also suppress dependencies on the permutations $\psi$ and $\psi'$. Furthermore, we focus on one generic summand and rely on arguments independent of the specific choice of permutations.
Using the above introduced graph analogy we observe that $\sum_{j'} \gamma_{k,j'} + \sum_{j} \beta_{k,j} =m$ as each node has exactly $m$ connections either within its own graph, the $\beta$'s, or with the other one, the $\gamma$'s, independently of the choice of permutations. Looking at the integral with respect to $s_k$ and pulling the kernels $x(t'-s'_{j'})$ and $ x(t-s_j)$, with the right exponents according to the calculation below, into the integral, we may apply the generalized H\"older inequality, as $ \sum_{j'} \f {\gamma_{k,j'}} {m}  +  \sum_{j} \f {\beta_{k,j}} {m}  =1$ and obtain,
\begin{align*}
&\int_{-\infty}^t d s_k \prod_{j'=1}^{k'} \vert x(t-s_k)^{\f {\gamma_{k,j'}} {m}} \vert \vert  x(t'-s_{j'}')\vert^{\f {\gamma_{k,j'}} {m}}   \vert s_k - s_{j'}'\vert^{\gamma_{k,j'} (2H_0 -2)} \prod_{q=1}^{k-1} \vert x(t-s_k)^{\f {\beta_{k,q}} {m}} \vert \vert x(t-s_q)^{\f {\beta_{k,q}} {m}} \vert \vert s_k - s_q \vert^{\beta_{k,q} (2H_0 -2)}  \\
&\leq \prod_{j'=1}^{k'} \left( \int_{-\infty}^t d s_k  \vert x(t-s_k) x(t'-s_{j'}') \vert \vert s_k - s_{j'}'\vert^{m (2H_0 -2)} \right)^{\f {\gamma_{k,j'}} {m}}  \prod_{q=1}^{k-1}  \left( \int_{-\infty}^t d s_k \vert x(t-s_k)  x(t-s_q) \vert  \vert s_k - s_q \vert^{m (2H_0 -2)} \right)^{\f {\beta_{k,q}} {m}}.
\end{align*}
Philosophically, H\"olders inequality enables us to factorize the integral with respect to $s_k$ by trading the terms $ \vert x(t'-s_{j'}')\vert^{\f {\gamma_{k,j'}} {m}} \vert $ for $\left( \int_{-\infty}^t d s_k  \vert x(t-s_k) x(t'-s_{j'}') \vert \vert s_k - s_{j'}'\vert^{m (2H_0 -2)} \right)^{\f {\gamma_{k,j'}} {m}}$, the terms  $\vert  x(t-s_q)\vert^{\f {\beta_{k,q}} {m}}  $ for $\left( \int_{-\infty}^t d s_k  \vert x(t-s_k)  x(t-s_q) \vert  \vert s_k - s_q \vert^{m (2H_0 -2)} \right)^{\f {\beta_{k,q}} {m}}$ and getting rid of the terms $\vert s_k - s_{j'}'\vert^{\gamma_{k,j'} (2H_0 -2)} $ as well as $\vert s_k - s_{j}\vert^{\beta_{k,j} (2H_0 -2)} $. Note that the exponents stay the same, meaning we trade a term with exponent ${\f {\gamma_{k,j'}} {m}}$ for one which again has exponent ${\f {\gamma_{k,j'}} {m}}$ and similarly for the $\beta$ terms. Hence, looking at the $s'_{k'}$ term we find the same situation as for the $s_k$ terms, at least exponent wise and we may repeat the procedure all over again.

When looking at the graph picture this procedure is a bit like cutting with a mince knife. We start at the top of the first graph and "cut" $s_k$ out of the picture, then, we move to the other side and cut $s'_{k'}$. As when cutting herbs we now again move to the other side with our knife and tackle $s_{k-1}$ followed by $s'_{k'-1}$ and so on until we have cut all edges and,  hopefully, what remains is less entangled and easier to deal with.
Another useful visualisation is that as soon as the knife hits a node with $m$ edges it splits into $m$ nodes which each has one edge. This alone does not help as at first the other end of the edge is possibly still entangled, but as soon as the mince knife hits the other node we obtain a simple integral, corresponding to that edge. In the end we obtain a product over all these integrals.

Picking up our mince knife, if we look at the terms which include $s'_{k'}$, again after pulling in other kernels with the right exponents,	we end up with,
\begin{align*}
&\int_{-\infty}^{t'} d s'_{k'} \prod_{j=1}^{k-1} \vert x(t'-s'_{k'})\vert^{1- \f {\gamma_{k',k}} {m}} \vert x(t-s_{j})\vert^{\f {\gamma_{k',j}} {m}}  \vert s'_{k'} - s_{j}\vert^{\gamma_{k',j} (2H_0 -2)} \prod_{q'=1}^{k'-1} \vert x(t'-s'_{k'})\vert^{\f {\beta_{k',q'}} {m}} \\ & \vert x(t'-s'_{q'})\vert^{\f {\beta_{k',q'}} {m}}  \vert s'_{k'} - s'_{q'} \vert^{\beta'_{k',q'} (2H_0 -2)} \left( \int_{-\infty}^t d s_k  \vert x(t-s_k) x(t'-s_{k'}') \vert\vert s_k - s_{k'}'\vert^{m (2H_0 -2)} \right)^{\f {\gamma_{k,k'}} {m}}. 
\end{align*}
Note that we have already used up some part of the $x(t'-s'_{k'})$ kernel leading to the ${1- \f {\gamma_{k',k}} {m}}$ exponent. As indicated above the sum of the exponents has the same structure, hence we obtain,
\begin{align*}
&\int_{-\infty}^{t'} d s'_{k'} \prod_{j=1}^{k-1} \vert x(t'-s'_{k'})\vert^{1- \f {\gamma_{k',k}} {m}}  \vert x(t-s_{j})\vert^{\f {\gamma_{k',j}} {m}} \vert s'_{k'} - s_{j}\vert^{\gamma_{k',j} (2H_0 -2)} \prod_{q'=1}^{k'-1} \vert x(t'-s'_{k'})\vert^{\f {\beta_{k',q'}} {m}}  \vert s'_{k'} - s'_{q'} \vert^{\beta'_{k',q'} (2H_0 -2)}\\ &\left( \int_{-\infty}^t d s_k \vert x(t-s_k) x(t'-s_{k'}') \vert \vert s_k - s_{k'}'\vert^{m (2H_0 -2)} \right)^{\f {\gamma_{k,k'}} {m}} \\
&\leq \left( \int_{-\infty}^{t'} ds'_{k'} \int_{-\infty}^{t} ds_k \vert  x(t-s_k) x(t'-s_{k'}') \vert \vert s_k - s_{k'}'\vert^{m (2H_0 -2)} \right)^{\f {\gamma_{k,k'}} {m}}\\ &\prod_{j=1}^{k-1} \left( \int_{-\infty}^{t'} ds'_{k'} \vert x(t'-s'_{k'}) x(t-s_j) \vert \vert s_j - s_{k'}'\vert^{m (2H_0 -2)} \right)^{\f {\gamma_{k',j}} {m}} \prod _{q'=1}^{k'-1} \left( \int_{-\infty}^{t'} ds'_{k'} \vert  x(t'-s'_{k'}) \vert  \vert s'_{k'} - s'_{q'} \vert^{m (2H_0 -2)} \right)^{\f {\beta'_{k',q'}} {m}}.
\end{align*}
The term 
$$\int_{-\infty}^{t'} ds'_{k'} \int_{-\infty}^{t} ds_k  \vert x(t-s_k) x(t'-s_{k'}')  \vert \vert s_k - s_{k'}'\vert^{m (2H_0 -2)}
$$
now obeys the decay imposed in Assumption \ref{assumption-decay-correlation-kernel} and we obtained it to the exponent $\gamma_{k,k'}$. Cutting of $s_{k-1}$ we  obtain, in particular, the term,
$$\left( \int_{[-\infty,t]^2}  ds_k ds_{k-1}  \vert x(t-s_k) x(t-s_{k-1}) \vert \vert s_k - s_{k-1}\vert^{m (2H_0 -2)} \right)^{\beta_{k,k-1} (2H_0-2)}.
$$
This term behaves like a power of $\Vert x \Vert_{\vert \mathcal{H} \vert }$, thus, the term only contributes a constant and we do not gain any decay from it.
Summarizing, we end up with the following picture, edges within a graph do not lead to any decay, however, the ones between the graphs do with a rate given by powers of $$ \int_{-\infty}^{t'} ds'_{k'} \int_{-\infty}^{t} ds_k  \vert x(t-s_k) x(t'-s_{k'}')  \vert \vert s_k - s_{k'}'\vert^{m (2H_0 -2)}.$$
As the $\gamma$'s represent the edges between the graphs and there are exactly $d$ of them, $\sum_{j=1}^k \sum_{j'=1}^{k'} \gamma_{j,j'} =d$, we obtain,
\begin{align*}
\Big \vert \E \left[ I_d\left(g^{k,r}_t \right) I_d\left(g^{k',r'}_{t'}\right)  \right] \Big \vert &\lesssim
\prod_{j =1}^{k} \prod_{j' =1}^{k'} \left( \int_{-\infty}^{t'} dv \int_{-\infty}^{t} du   \vert x(t-u) x(t'-v) \vert \vert u - v\vert^{m (2H_0 -2)} \right)^{ \f  {\gamma_{j,j'}} {m} }\\
&= \left( \int_{-\infty}^{t'} dv \int_{-\infty}^{t} du  \vert x(t-u) x(t'-v) \vert \vert u - v\vert^{m (2H_0 -2)} \right)^{\f d m}\\
&\lesssim  \left( 1 \wedge \vert t-t' \vert^{m (2H_0 -2)} \right)^{ \f d m} \\
&\lesssim 1 \wedge \vert t-t' \vert^{ (2H_0 -2)d},
\end{align*}
by Assumption \ref{assumption-decay-correlation-kernel}.

It is left to prove the bound on the proportionality constant.
We picked up constants corresponding to powers of $ \Vert x \Vert_{\vert \mathcal{H} \vert} $ from the edges between each graph. As there are at most $(k+k')m$ factors we can bound them by $(1+ \Vert x \Vert_{ \vert \mathcal{H} \vert })^{(k+k')m}$.
Next, we bound the amount of possible contractions $r$ such that $\delta(k,r) =d$, the case $\delta(k',r')$ can be bounded analogously. As each single contraction can at most have rank $m$ we can control this quantity by $m^k$. The constants $C_3(\psi,r,H_0)$ can be bounded by $ (C_2(H_0)+2)^{(k+k')}$. Set $\mathfrak{L}= C_2(H_0)+3+\Vert x\Vert_{\mathcal{H}} + K(H,m)$ to combine this constants and add in the normalizations $K(H,m)$.
Next, the constants $C_1(r,k,m)$, by simply bounding $r_j!$ by $m!$ and each binomial coefficient by $k^m$, can be bounded by $(m!)^k k^m$.
Finally, we deal with the constant $d!$ obtained from using the Wiener-It\^o isometry, cf. Equation \ref{eqaution-L2-isometrie}. As  $d \leq (k \wedge k')m$  we may bound $d!$ by $\sqrt{(km)!(k'm)!}$ and add the multiplicative constant $\f {((k+k')m)^2} {d^2}$, in order to have a summable decay in $d$, to conclude the proof.

\end{proof}

\begin{lemma}\label{lemma-variance-asymptotics}
Given $h^{d,k}_t$ as above, then, the following holds,
\begin{equation} 
 \f{1}{C_4(k.k',m,d)}\int_0^{\f T \epsilon} \int_0^{\f T \epsilon} \Big \vert \E \left[  I_d(h^{d,k}_t)  I_d(h^{d,k'}_{t'})  \right] \Big \vert dt dt' \\
\lesssim 
\begin{cases}
\f T \epsilon  ,  \quad  &\hbox {if} \quad  H^*(d)< \f 1 2, \\
\f T \epsilon \vert \ln( \epsilon) \vert , \quad  &\hbox {if} \quad H^*(d) = \f 1 2, \\
\left(  \f T \epsilon\right) ^{(2H_0 -2)d +2},  \quad &\hbox {if} \quad H^*(d) > \f 1 2.
\end{cases}
\end{equation}	
 
\end{lemma}
\begin{proof}

Using Lemma \ref{wachstum-konstanten}, we compute,
\begin{align*}
\f{1}{C_4(k.k',m,d)} \int_0^{\f T \epsilon} \int_0^{\f T \epsilon} \Big \vert \E \left[  I_d(h^{d,k}_t)  I_d(h^{d,k'}_{t'})  \right] \Big \vert dt dt' &\lesssim\int_{[0,\f T \epsilon]^2}  1 \wedge \vert t-t' \vert^{ (2H_0 -2)d} dt dt'\\ &\lesssim \f  T \epsilon  \int_0^{\f T \epsilon}  1 \wedge \vert u \vert^{ (2H_0 -2)d} du.
\end{align*}
Depending on the exponent $(2H_0 -2)d$, the above integral is either finite or diverges with rate $\epsilon^{(2H_0-2)d+1}$, resulting in the proclaimed rate. Overall we obtain,

\begin{equation} 
\f{1}{C_4(k.k',m,d)}\E \left[ \left(\int_0^{\f T \epsilon} I_d(h_t) dt\right)^2 \right] \\
\lesssim 
\begin{cases}
 \f T \epsilon  ,  \quad  &\hbox {if} \quad  (2H_0 -2)d < -1 ,\\
\f T \epsilon \vert \ln( \epsilon) \vert , \quad  &\hbox {if} \quad (2H_0 -2)d = -1 ,\\
\left(  \f T \epsilon\right) ^{(2H_0 -2)d +2},  \quad &\hbox {if} \quad (2H_0 -2)d < -1,
\end{cases}  
\end{equation}		
concluding the proof.
\end{proof}

\begin{convention}
We call building blocks $g^{k,r}_t$ short range dependent if $H^*(\delta(k,r) ) < \f 1 2$ and long range dependent in case $H^*(\delta(k,r) ) > \f 1 2$ . 
For a function $G$ we say it is short range dependent if all its building blocks are short range dependent and long range dependent in case it contains at least one long range dependent building block.
\end{convention}

\subsection{Tightness in H\"older spaces}\label{subsection-tightness-Hoelder}
In this section we apply Lemma $\ref{lemma-variance-asymptotics}$ as well as a basic hypercontractivity estimate to obtain tightness in H\"older spaces via a Kolmogorov type argument. In case the function $G$ is a finite polynomial the $L^2$ bound obtained above directly gives us an $L^p$ bound via hypercontractivity. In order to treat the infinite series case, we rely on the decay rate of the coefficients.
\begin{lemma}\label{lemma-kolmogorv-bound}
	 Given $G(X) = \sum_{k=0}^{\infty} c_{k} X^k$ such that $G$ has chaos rank $w \geq 1 $ with respect to $y_t$, $\vert c_k \vert \lesssim \f {1} {k!}$ and $H^*(w) \in (-\infty,1) \setminus\{\f 1 2\} $, then, for every $p>2$,
	$$ \Big \Vert \epsilon^{H^*(w) \vee \f 1 2} \int_{ \f{S}{\epsilon}}^{\f T \epsilon} G(y_s) ds  \Big \Vert_{L^p(\Omega)} \lesssim \vert T-S\vert^{H^*(w) \vee \f 1 2}.$$
\end{lemma}
\begin{proof}
	Firstly,  by  stationarity of $y_t$ we may restrict our analysis to the case $S=0$. 
	Secondly, we want to remark at this point that by assumption the chaos rank of $G$ is greater than $1$, hence, its average with respect to the distribution of $y_t$ is $0$.
	Thirdly, we recall the following hypercontractivity  estimate, for a random variable $X$ belonging to the $d^{th}$ Wiener chaos the following  holds true for all $p>2$, c.f. \cite{Nualart},
	$$ \Vert X \Vert_{L^p(\Omega)} \leq \left( p-1\right)^{ \f d 2} \Vert X \Vert_{L^2(\Omega)}.
	$$
	Using this bound and the triangle inequality, we compute,
	\begin{align*}
	\Big \Vert  \epsilon^{H^*(w) \vee \f 1 2} \int_{0}^{\f T \epsilon} G(y_s) ds \Big \Vert_{L^p(\Omega)} &= \Big \Vert  \sum_{d=w}^{\infty} \sum_{k=0}^{\infty} c_k \epsilon^{H^*(w) \vee \f 1 2} \int_0^{\f T \epsilon}  I_d(h^{d,k}_{t}) dt \Big \Vert_{L^p(\Omega)}\\
	&\leq  \sum_{d=w}^{\infty} \Big \Vert \sum_{k=0}^{\infty} c_k  \epsilon^{H^*(w) \vee \f 1 2} \int_0^{\f T \epsilon}  I_d(h^{d,k}_{t}) dt \Big \Vert_{L^p(\Omega)}\\
	&\lesssim \sum_{d=w}^{\infty} \left(p-1\right)^{\f d 2} \Big \Vert \sum_{k=0}^{\infty} c_k   \epsilon^{H^*(w) \vee \f 1 2} \int_0^{\f T \epsilon}  I_d(h^{d,k}_{t}) dt \Big \Vert_{L^2(\Omega)}.
	\end{align*}
	Now, we observe that the polynomial $(y_t)^k$ has contributions only in the chaoses up to order $km$,
	hence,  using Lemma \ref{lemma-variance-asymptotics},
\begin{align*}
&\sum_{d=w}^{\infty} \left(p-1\right)^{\f d 2} \Big \Vert \sum_{k=0}^{\infty} c_k   \epsilon^{H^*(w) \vee \f 1 2} \int_0^{\f T \epsilon}  I_d(h^{d,k}_{t}) dt \Big \Vert_{L^2(\Omega)}\\
&= \sum_{d=w}^{\infty} \left(p-1\right)^{\f d 2} \sum_{k=0}^{\infty} \sum_{k'=0}^{\infty}  c_k   c_{k'} \epsilon^{2H^*(w) \vee 1} \int_0^{\f T \epsilon} \int_0^{\f T \epsilon} \Big \vert \E \left[ I_d(h^{d,k}_{t}) I_d(h^{d,k'}_{t'}) \right] \Big \vert dt dt' \\
&\leq \sum_{d=w}^{\infty} \left(p-1\right)^{\f d 2} \sum_{k=0}^{\infty} \sum_{k'=0}^{\infty}  c_k   c_{k'} \epsilon^{2H^*(w) \vee 1} \int_0^{\f T \epsilon} \int_0^{\f T \epsilon}  1 \wedge \vert t-t'\vert^{(2H_0-2)d}  dt dt' \\
&\leq \sum_{d=w}^{\infty} \left(p-1\right)^{\f d 2} \sum_{k=0}^{\infty} \sum_{k'=0}^{\infty}  c_k   c_{k'} \epsilon^{2H^*(w) \vee 1} \int_0^{\f T \epsilon} \int_0^{\f T \epsilon}  1 \wedge \vert t-t'\vert^{(2H_0-2)w}  dt dt' \\
&\lesssim \sum_{d=w}^{\infty} \left(p-1\right)^{\f d 2} \sum_{k=0}^{\infty} \sum_{k'=0}^{\infty}  c_k   c_{k'} C_4(k,k',m,d) \vert T \vert^{H^*(w) \vee \f 1 2} \\
&\leq  \sum_{k=0}^{\infty} \sum_{k'=0}^{\infty} \sum_{d=w}^{(k+k')m} \left(p-1\right)^{\f {(k+k')m} {2}}   c_k   c_{k'} C_4(k,k',m,d) \vert T \vert^{H^*(w) \vee \f 1 2}\\
&\lesssim \vert T\vert^{H^*(w) \vee \f 1 2},
\end{align*}
as $\sum_{k=0}^{\infty} \sum_{k'=0}^{\infty} \sum_{d=w}^{(k+k')m} \left(p-1\right)^{\f {(k+k')m} {2}}   c_k   c_{k'} C_4(k,k',m,d)< \infty,$ due to the decay condition imposed on the coefficients $c_k$.
\end{proof}
\begin{proposition}\label{proposition-hoelder}
	Fix $H \in ( \f 1 2, 1)$. Let, for each $j =1, \dots, N$,  a function of the form $G^j(X) = \sum_{k=0}^{\infty} c_{j,k} X^k$ such that $G^j$ has chaos rank $w^j$ with respect to $y_t$ and $ \vert c_{j,k} \vert \lesssim \f {1}{k!}$ be given.  Assume further that for each $j$, $w_j \in (-\infty,1) \setminus \{\f 1 2\}$. For $T \in [0,1]$ set
	$$
	\bar{G}^{j,\epsilon}_T = \epsilon^{H^*(w_j) \vee \f 1 2} \int_0^{\f T \epsilon} G^j(y_t) dt.
	$$
Then, 
$$\left( \bar{G}^{1,\epsilon}_T, \dots, \bar{G}^{N,\epsilon}_T \right),$$
is tight in $\C^{\gamma}\left( [0,1] , \R^N \right)$ for $\gamma \in (0, \f 1 2)$ in case there is at least on component such that $H^*(w_j)< \f 1 2$ and $\gamma \in \left(0, \min_{j=1, \dots, N} H^*(w_j)\right)$ otherwise.
\end{proposition}
\begin{proof}
By Lemma \ref{lemma-kolmogorv-bound} and Kolmogorv's Theorem each component is tight in $\C^{\gamma}$ for $\gamma \in \left(0, \f 1 2 \vee H^*(w_j) \right)$. Therefore, taking the minimum over these values, we may conclude the proof.
\end{proof}

\subsection{Short range dependent case}

In this section we establish the convergence of building blocks in finite dimensional distributions in the short range dependent setting. Thus, we deal with terms of the form $ \sqrt{\epsilon} \int_0^{\f T \epsilon }g^{k,r}_t dt$ such that $\delta(k,r)=d$ and $H^*(d) < \f 1 2$. Our main tool to prove this convergence is the  fourth-moment theorem, which is stated below.

\begin{theorem}\label{fourth-moment-theorem}[Fourth Moment Theorem]
Let $2 \leq L$, $1 \leq d_1 \leq \dots \leq d_L$ be fixed integers and $f^{j,\epsilon} \in L^2(\R^{d_j})$ for $1 \leq j \leq L$ be given. Then, under the condition that  $\lim_{\epsilon \to 0} \E \left[ I_{d_j} (f^{j,\epsilon}) I_{d_l}(f^{l,\epsilon}) \right] =  \Lambda_{j,l}$ exists for $ 1 \leq j,l \leq L$, the following are equivalent:
\begin{enumerate}
\item For $1 \leq j \leq L$  and $ p =1 , \dots, d_j,$
$$ \lim_{\epsilon \to 0} \Vert f^{j,\epsilon} \otimes_p f^{j,\epsilon} \Vert_{L^2(\R^{2d_j-2p},\lambda)} =0.$$
\item The vector $ \left( I_{d_1}(f^{1,\epsilon}), \dots, I_{d_L}(f^{L,\epsilon}) \right)$ converges in distribution to a $L$ dimensional Gau{\ss}ian vector with mean zero and covariance matrix $\Lambda$.
	\end{enumerate}
\end{theorem}
\begin{remark}
	Up to now we labelled our contraction numbers by $r=(r_1, \dots , r_k)$. In case we need to denote several of such vectors we make the convention that upper indices denote different vectors and subscripts the position within the vector.
\end{remark}
Next, we prove that the conditions necessary to apply the fourth moment theorem are indeed satisfied in our regime. In the short range dependent case we set, for $T \in [0,1]$, 
$$   
\bar{g}^{k,r,\epsilon}_{T} = \sqrt{\epsilon} \int_0^{\f T \epsilon} g^{k,r}_t dt.
$$

\begin{lemma}\label{covariance-convergence}
For each $k,k',r,r'$ such that  $\delta(k,r)=d$, $\delta(k',r')=d'$,  $H^*(d) < \f 1 2$, and $\bar{g}^{k,r,\epsilon}_{T}, \bar{g}^{k',r',\epsilon}_{T'}$ for $T,T' \in [0,1]$ defined as above, the following holds,
$$ \lim_{\epsilon \to 0} \E \left[ I_{d} \left(\bar g^{k,r,\epsilon}_{T}\right) I_{d'}\left(\bar g^{k',r',\epsilon}_{T'}\right) \right]= \delta_{d,d'}
2 (T \wedge T') \int_{0}^{\infty } 
\E \left[ I_d\left(g^{k,r}_{u} \right) I_{d}\left(g^{k',r'}_{0}\right)  \right] du$$
\end{lemma}
\begin{proof}
Firstly, if $d \not = d'$ the expression is $0$ by orthogonality of distinct Wiener chaoses. Hence, we assume $d=d'$  and w.l.o.g $ T \leq T'$ from now on. In this case 
\begin{align*}
\E \left[  I_{d} \left(\bar g^{k,r,\epsilon}_{T}\right) I_{d}\left(\bar g^{k',r',\epsilon}_{T'}\right)   \right] &= \epsilon \int_0^{\f  T \epsilon } \int_0^{\f {T'} {\epsilon} } 
\E \left[ I_d\left(g^{k,r}_t \right) I_d\left(g^{k',r'}_{t'}\right)  \right] dt' dt\\ 
&= \epsilon \int_0^{\f  T \epsilon } \int_0^{\f {T} {\epsilon} } 
\E \left[ I_d\left(g^{k,r}_t \right) I_d\left(g^{k',r'}_{t'}\right)  \right] dt' dt  + 
\epsilon \int_0^{\f  T \epsilon } \int_{\f  T \epsilon }^{\f {T'} {\epsilon} } 
\E \left[ I_d\left(g^{k,r}_t \right) I_d\left(g^{k',r'}_{t'}\right)  \right] dt' dt
\end{align*}
By Lemma \ref{wachstum-konstanten} and the change of varibales $u= t-t'$, we obtain for the second term,
\begin{align*}
\epsilon \Big \vert \int_0^{\f  T \epsilon } \int_{\f  T \epsilon }^{\f {T'} {\epsilon} } 
\E \left[ I_d\left(g^{k,r}_t \right) I_d\left(g^{k',r'}_{t'}\right)  \right] dt' dt \Big \vert &\lesssim \epsilon \int_0^{\f  T \epsilon } \int_{\f  T \epsilon }^{\f {T'} {\epsilon} } 1 \wedge \vert t- t' \vert^{2H_0-2} dt' dt \\
& \lesssim \int_{\f  T \epsilon }^{\f {T'} {\epsilon} } 1 \wedge u^{2H_0-2} du.
\end{align*}
By assumption $H^*(d) < \f 1 2 $, leading to $ 2H_0 -2 < -1$, hence, the above term converges to $0$ as $\epsilon \to 0$ in case $T>0$. If $T=0$, however, the whole expression equals $0$ independent of $\epsilon>0$.
Now we deal with the first term.
As $\E \left[ I_d\left(g^{k,r}_t \right) I_d\left(g^{k',r'}_{t'}\right)  \right] = \E \left[      I_d\left(g^{k,r}_{t-t'}\right) I_d\left(g^{k',r'}_{0}\right) \right]$ we obtain, again via the change of variables $u= t-t'$,
\begin{align*}
\epsilon \int_0^{\f  T \epsilon } \int_0^{\f {T} {\epsilon} } 
\E \left[ I_d\left(g^{k,r}_t \right) I_d\left(g^{k',r'}_{t'}\right)  \right] dt' dt  &=  \epsilon \int_0^{\f  T \epsilon } \int_0^{\f {T} {\epsilon} } 
\E \left[ I_d\left(g^{k,r}_{t-t'} \right) I_d\left(g^{k',r'}_{0}\right)  \right] dt' dt\\
&=   2T  \int_{0}^{\f  T \epsilon } \f{T-\epsilon u} {T}
\E \left[ I_d\left(g^{k,r}_{u} \right) I_d\left(g^{k',r'}_{0}\right)  \right] du \\
&\to 2 T \int_{0}^{\infty } 
\E \left[ I_d\left(g^{k,r}_{u} \right) I_d\left(g^{k',r'}_{0}\right)  \right] du,
\end{align*}
when $\epsilon \to 0$ by dominated convergence.
\end{proof}

\begin{lemma}\label{contractions-go-to-0}
For each $k,r$ such that  $\delta(k,r)=d$, $T \in [0,1]$, $H^*(d)< \f 1 2$, $p \leq d-1$, and $\bar{g}^{k,r,\epsilon}_{T}$ as above,
$$
\lim_{\epsilon \to 0} \Vert \bar{g}^{k,r,\epsilon}_{T} \otimes_p \bar{g}^{k,r,\epsilon}_{T} \Vert_{L^2\left(\R^{2d-2p},\lambda\right)} =0.
$$
\end{lemma}
\begin{proof}
We compute,
\begin{align*}
&\Vert \bar{g}^{k,r,\epsilon}_{T} \otimes_p \bar{g}^{k,r,\epsilon}_{T} \Vert_{L^2(\R^{2d-2p})}^2 = \epsilon^2 \int_{\R^{d-p}} d^{d-p} \xi_1  \int_{\R^{d-p}} d^{d-p} \xi_2 \Big \vert \int_{\R^p} d^p u \int_0^{\f T \epsilon} g^{k,r}_t(\xi_1,u) dt   \int_0^{\f T \epsilon} g^{k,r}_{t'}(\xi_2,u) dt'  \Big \vert^2 \\
&= \epsilon^2 \int_{\R^{d-p}} d^{d-p} \xi_1  \int_{\R^{d-p}} d^{d-p} \xi_2 \int_{\R^p} d^p u  \int_0^{\f T \epsilon} g^{k,r}_t(\xi_1,u) dt   \int_0^{\f T \epsilon} g^{k,r}_{t'}(\xi_2,u) dt'  \int_{\R^{p}} d^{p} v \int_0^{\f T \epsilon} g^{k,r}_{\tau} (\xi_1,v) d\tau   \int_0^{\f T \epsilon} g^{k,r}_{\tau'}(\xi_2,v) d\tau'  .
\end{align*} 
The integrals of the form
$$ 
\int_{\R^p} d^p u  \int_0^{\f T \epsilon} g^{k,r}_t(\xi_1,u) dt   \int_0^{\f T \epsilon} g^{k,r}_{t'}(\xi_2,u) dt' 
$$
are just further contractions of $f_t$ with itself and can be computed as in section \ref{section-decomposition}. We obtain, up to the constants $C_2(H_0)$,
\begin{align*}
\int_{\R^p} d^p u  \int_0^{\f T \epsilon} g^{k,r}_t(\xi_1,u) dt   \int_0^{\f T \epsilon} g^{k,r}_{t'}(\xi_2,u) dt' &= \sum_{\psi_1,\psi_2 \in \mathcal{S}_r} \int_0^{ \f T \epsilon}  \int_0^{ \f T \epsilon} dt dt'  \int_{[-\infty,t]^k} d^k s \int_{[-\infty,t']^{k}} d^k s' \prod_{j=1}^k \prod_{j'=1}^k   \\
&x(t-s_{j}) x(t'-s'_{j'}) \prod_{q=1}^{j-1} \vert s_{j} - s_{q} \vert^{\beta_{j,q}(\psi_1) (2H_0-2) }  \prod_{q'=1}^{j'-1} \vert s'_{j'} - s'_{q'} \vert^{ \beta_{j',q'}(\psi_2) (2H_0-2)}  \\
& \vert s_j - s'_{j'} \vert^{\gamma_{j,j'}(\psi_1,\psi_2) (2H_0-2) } \prod_{l=1}^{M_j(\psi_1)}  \vert s_j -\xi_{j,l} \vert^{H_0-\f 3 2} \prod_{l'=1}^{M_{j'}(\psi_2)} \vert s'_{j'} -\xi_{j',l'} \vert^{H_0-\f 3 2} , 
\end{align*}
and similarly for
\begin{align*}
\int_{\R^{p}} d^{p} v \int_0^{\f T \epsilon} g^{k,r}_{\tau} (\xi_1,v) d\tau   \int_0^{\f T \epsilon} g^{k,r}_{\tau'}(\xi_2,v) d\tau'.
\end{align*}
Computing the integrals with respect to $\xi_1$ and $\xi_2$ we again score constants $C_2(H_0)$ which we can neglect and end up with,
$$
\sum_{\psi_1,\psi_2,\psi_3,\psi_4 \in \mathcal{S}_r} \int_{[0, \f T \epsilon]^4 }d^4t  \prod_{j=1}^4 \int_{[-\infty,t_j]^k} d^{4k}s \vert x(t_j - s_{j,l}) \vert \prod_{j'=1}^4 \prod_{ l=1,l'=1 }^k \vert s_{j,l} - s_{j',l'} \vert^{A^{j,l}_{j',l'}(\psi_{j}, \psi_{j'}) (2H_0-2)},
$$
where we make the convention $t= t_1, t' = t_2 , \tau = t_3, \tau' = t_4$ and
$$
A^{j,l}_{j',l'}(\psi_j, \psi_{j'}) = \begin{cases}
\beta_{l,l'}(\psi_j) &\text{if }   j=j',\\
\gamma_{l,l'}(\psi_j, \psi_{j'}) &\text{if } j \not = j'.
\end{cases}
$$
Now, we have to consider four graphs. Nevertheless, as in the discussions above we focus on one particular choice of permutations and rely on arguments which only depend on $k,r,m$ and $H_0$, thus, we suppress dependencies on the $\psi$'s from now on. Hence, the term we need to deal with is given by,
\begin{align*}
\epsilon^2 \int_{[0, \f T \epsilon]^4 }d^4t  \prod_{j=1}^4 \int_{[-\infty,t_j]^d} d^{4k}s \vert x(t_j - s_{j,l}) \vert  \prod_{j'=1}^4 \prod_{ l=1,l'=1 }^k \vert s_{j,l} - s_{j',l'} \vert^{A^{j,l}_{j',l'} (2H_0-2)},
\end{align*}
where $A^{j,l}_{j',l'}$ depends on the four permutations we choose when picking graphs and denotes how many edges run between the $l^{th}$ node in the $j^{th}$ graph to the $l'^{th}$ in the $j'^{th}$ one.
Keeping the graph view, the following illustration shows the entanglements between the four objects, where $B_{j,j'} = \sum_{l=1 ,l'=1}^k {A^{j,l}_{j',l'}}$ is the number of all edges going from graph $j$ to graph $j'$.
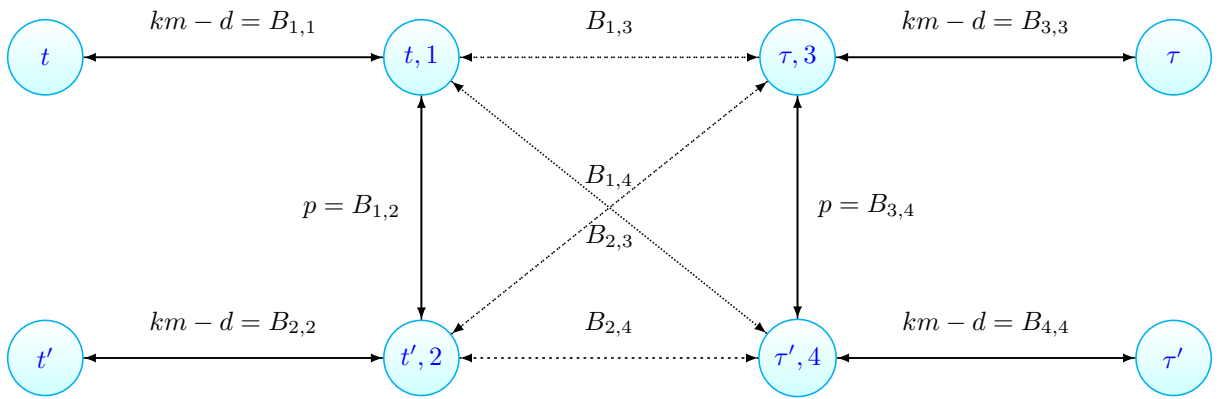
\begin {center}
\begin {tikzpicture}[-latex ,auto ,node distance =4 cm and 5cm ,on grid ,
semithick ,
state/.style ={ circle ,top color =white , bottom color = processblue!20 ,
	draw,processblue , text=blue , minimum width =1 cm}]
\node[state] (A) {$t ,1 $};
\node[state] (B) [ below =of A] {$t' ,2  $};
\node[state] (C) [ right=of A] {$\tau, 3  $};
\node[state] (D) [ right =of B] {$\tau' ,4 $};
\node[state] (A') [left= of A] {$t$};
\node[state] (B') [ left =of B] {$t'$};
\node[state] (C') [ right=of C] {$\tau$};
\node[state] (D') [ right =of D] {$\tau'$};
\path (A)  edge  node[above =0.15 cm] {$km-d = B_{1,1}$} (A');
\path (A') edge (A);
\path (B) edge node[above =0.15 cm] {$km-d = B_{2,2}$}  (B');
\path (B') edge (B);
\path (C) edge  node[above =0.15 cm] {$km-d = B_{3,3}$} (C');
\path (C') edge (C);
\path (D) edge  node[above =0.15 cm] {$km-d  = B_{4,4}$} (D');
\path (D') edge (D);
\path (A) edge [dotted]  node[above =0.15 cm] {$B_{1,3}$}   (C);
\path (C) edge  [dotted] (A);
\path (D) edge (C);
\path (C) edge node[right =0.15 cm] {$p = B_{3,4}$} (D);
\path (D) edge [dotted] (B);
\path (B) edge  [dotted] node[above =0.15 cm] {$B_{2,4}$}  (D);
\path (A) edge node[left =0.15 cm] {$p = B_{1,2}  $} (B);
\path (B) edge (A);
\path (D) edge [dotted] (A);
\path (A) edge  [dotted] node[above  =0.10 cm] {$B_{1,4}$}  (D);
\path (C) edge [dotted] node[below  =0.10 cm] {$B_{2,3}$} (B);
\path (B) edge  [dotted] (C);
\end{tikzpicture}
\end{center}

As previously we have some algebraic constraints, namely,
$B_{1,3} + B_{1,4} = d-p$ and $B_{2,3}+B_{2,4} = d-p$.

By the same arguments as in section \ref{section-variance-computation} we can iteratively apply the generalized H\"older inequality and bound the above integrand by
$$ \prod_{j \geq j'} \left( \int_{-\infty }^{t_j} \int_{-\infty}^{t_{j'}} \vert x(t_j -s) x(t_{j'}-r) \vert  \vert s-r \vert^{m(2H_0-2)} dr ds \right)^{ \f {B_{j,j'}} {m} }. 
$$
Terms for which $j=j'$ give rise to constants as they represent connections within a graph. For the other terms we obtain the same decay as in Lemma \ref{lemma-variance-asymptotics},
$$
\left( \int_{ -\infty}^{t_j} \int_{-\infty}^{t_{j'}} \vert x(t_j -s) x(t_{j'}-r) \vert \vert s-r \vert^{m(2H_0-2)} dr ds \right)^{ \f {B_{j,j'}} {m} } \leq 1 \wedge  \vert t_j - t_{j'} \vert ^{(2H_0-2)B_{j,j'}}.
$$
Hence, we obtain, for $\rho(s) = 1 \wedge \vert s \vert^{(2H_0-2)m}$,
\begin{align*}
&\int_{[0, \f T \epsilon]^4 }d^4t  \prod_{j=1}^4 \int_{[-\infty,t_j]^k} d^{4k}s \vert x(t_j - s_{j,l}) \vert  \prod_{j'=1}^4 \prod_{ l=1,l'=1 }^k \vert s_{j,l} - s_{j',l'} \vert^{A^{j,l}_{j',l'} (2H_0-2)}\\
&\lesssim  \int_{[0, \f T \epsilon]^4} d^4 t \prod_{j \geq j'} \left( \int_{ -\infty}^{t_j} \int_{-\infty}^{t_{j'}} \vert x(t_j -s) x(t_{j'}-r) \vert \vert s-r \vert^{m(2H_0-2)} dr ds \right)^{ \f {B_{j,j'}} {m} } \\
& \lesssim \int_{[0, \f T \epsilon]^4} d^4t  \rho(t_1-t_2)^p \rho(t_3 - t_4)^p \rho(t_1-t_3)^{B_{1,3}} \rho(t_1-t_4)^{B_{1,4}} \rho(t_2-t_3)^{B_{2,3}} \rho(t_2-t_4)^{B_{2,4}}.
\end{align*}

As $B_{1,3} + B_{1,4} = d-p$ we obtain  $ \rho(t_1-t_3)^{B_{1,3}} \rho(t_1-t_4)^{B_{1,4}} \leq \rho(t_1-t_3)^{d-p} + \rho(t_1-t_4)^{d-p}$ and similarly $ \rho(t_2-t_3)^{B_{2,3}} \rho(t_2-t_4)^{B_{2,4}} \leq \rho(t_2-t_3)^{d-p} + \rho(t_2-t_4)^{d-p}$. Thus,
\begin{align*}
&\int_{[0, \f T \epsilon]^4} d^4t \rho(t_1-t_2)^p \rho(t_3 - t_4)^p \rho(t_1-t_3)^{B_{1,3}} \rho(t_1-t_4)^{B_{1,4}} \rho(t_2-t_3)^{B_{2,3}} \rho(t_2-t_4)^{B_{2,4}} \\
&\leq \int_{[0, \f T \epsilon]^4} d^4t \rho(t_1-t_2)^p \rho(t_3 - t_4)^p [  \rho(t_2-t_3)^{d-p} + \rho(t_2-t_4)^{d-p} ] [ \rho(t_1-t_3)^{d-p} + \rho(t_1-t_4)^{d-p}].
\end{align*}
We just treat the term $  \rho(t_1-t_2)^p \rho(t_3 - t_4)^p  \rho(t_2-t_3)^{d-p}  \rho(t_1-t_3)^{d-p}$ as the others are analogous. Applying the same procedure once more we get $\rho(t_1-t_2)^p  \rho(t_1-t_3)^{d-p} \leq  \rho(t_1-t_2)^{d} +   \rho(t_1-t_3)^{d}.$
Now,
\begin{align*}
& \int_{[0, \f T \epsilon]^4} d^4t  \rho(t_1-t_2)^p \rho(t_3 - t_4)^p  \rho(t_2-t_3)^{d-p}  \rho(t_1-t_3)^{d-p}\\
&\leq  \int_{[0, \f T \epsilon]^4} d^4t \rho(t_3 - t_4)^p  \rho(t_2-t_3)^{d-p}  [  \rho(t_1-t_2)^{d} +   \rho(t_1-t_3)^{d}].\\
\end{align*}
Again we only treat the case $\rho(t_3 - t_4)^p  \rho(t_2-t_3)^{d-p} \rho(t_1-t_2)^{d}$.
We finally obtain,
\begin{align*}
\int_{[0, \f T \epsilon]^4} d^4t \rho(t_3 - t_4)^p  \rho(t_2-t_3)^{d-p} \rho(t_1-t_2)^{d} &\leq \int_{\R} du \rho( \vert u \vert)^d  \int_{[0, \f T \epsilon]^3} d^3 t \rho(t_3 - t_4)^p  \rho(t_2-t_3)^{d-p}\\
&\leq  \f{1} {\epsilon} \int_{\R} du_1 \rho(\vert u_1 \vert )^d   \int_{\R} du_2 \rho(\vert u_2 \vert )^p    \int_{\R} du_3 \rho(\vert u_3 \vert)^{d-p} .
\end{align*}
Now, the first integral is finite as by assumption $H^*(d) < \f 1 2$. Furthermore, $\int_{\R} \rho(\vert t_3' \vert )^p dt_3' \lesssim \epsilon^{-(2H_0-2)p-1}$ and $\int_{\R} \rho(\vert t_2 \vert)^{d-p} dt_2 \lesssim \epsilon^{-(2H_0-2)(d-p) -1}$, cf. the proof of Lemma
\ref{lemma-variance-asymptotics}.  Thus, the expression is of order $\epsilon^{-(2H_0-2)d -3}$ and as $H^*(d) < \f 1 2$ this expression is of order $o(\epsilon^2)$, which concludes the proof.
\end{proof}

Now we have all tools necessary to conclude the proclaimed convergence in finite dimensional distributions for our building blocks. 
\begin{proposition}\label{proposition-joint-wiener-building-blocks-srd}
Given the process $ \left( I_{d_{1}}(\bar{g}^{k_1,r^1,\epsilon}_{T}) , \dots , I_{d_{L}}(\bar{g}^{k_L,r^L,\epsilon}_{T} ) \right)$, for which $d_j= \delta(k_j,r^j)$ such that for each $1 \leq j \leq L$,  $H^*(d_j) < \f 1 2$, then,
$$ 
\left( I_{d_{1}}(\bar{g}^{k_1,r^1,\epsilon}_{T}) , \dots, I_{d_{L}}(\bar{g}^{k_L,r^L,\epsilon}_{T} ) \right)  \to \left( W_{1} , \dots, W_{L} \right),
$$
in finite dimensional distributions, where $\left( W_{1} , \dots, W_{L} \right)$ is a multidimensional Wiener process with covariance matrix
\begin{align*}
\E \left[ W_{j}(1) W_{l}(1) \right] = \Lambda_{j,l}&=  \lim_{\epsilon \to 0} \E \left[ I_{d_j}(\bar{g}^{k_j,r^j,\epsilon}_{1})  I_{d_l}(\bar{g}^{k_l,r^l,\epsilon}_{1}) \right]\\
&=
\delta_{d_j,d_l} 2 \int_{0}^{\infty } 
\E \left[ I_{d_j}\left(g^{k_j,r^j}_{u} \right) I_{d_l}\left(g^{k_l,r^l}_{0}\right)  \right] du.
\end{align*}
\end{proposition}
\begin{proof}
By the Fourth Moment Theorem \ref{fourth-moment-theorem} and an application of the Cramer Wold Theorem  \ref{Cramer-Wold-theorem} the claim follows by  Lemma \ref{covariance-convergence} and Lemma \ref{contractions-go-to-0}.
\end{proof}

\subsection{Long range dependent case}
In this section we establish the convergence of building blocks in finite dimensional distributions in the long range dependent setting. Thus, we deal with terms of the form $ \epsilon^{H^*(d)} \int_0^{\f T \epsilon }g^{k,r}_t dt$ such that $\delta(k,r)=d$ and $H^*(d) > \f 1 2$. Our main tool to prove this convergence is $L^2$ kernel convergence.

\subsubsection{Spectral representation}\label{subsubsection-spectral-measure}

In this section we  use a connection between iterated Wiener integrals and a Gau{\ss}ian complex-valued random spectral measure $\hat{B}$ such that for all Borel sets the following holds, $\E \left[ \hat{B}(A) \right] = 0$, $ \E \left[ \hat{B}(A_1) \hat{B}(A_2)  \right] = \lambda( A_1 \cap A_2)$, $  \hat{B}(A)  = \overline{ \hat{B}(-A) }$ and for disjoint Borel sets $\hat{B}( \cup_{j=1}^n A_j) = \sum_{j=1}^n \hat{B}(A_j)$.
The random variables $\Re \left(\hat{B}(A)\right) $ and $\Im\left(  \hat{B}(A)\right)$ are independent Gau{\ss}ians with zero mean and variance $\f {\lambda(A)} {2}$. 
One can now define multiple Wiener integrals for even square integrable complex valued symmetric functions, meaning for $\hat{\xi_j} \in \R$,
$f(\hat{\xi}_1, \hat{\xi_2}, \dots, \hat{\xi_d}) = \overline { f(-\hat{\xi}_1, -\hat{\xi_2}, \dots, -\hat{\xi_d}) }$
and 
$ \int_{\R^d} d^d\hat{\xi} \vert f(\hat{\xi}_1,  \dots, \hat{\xi_d}) \vert^2 < \infty.$ Denoting the space of these functions by $\mathcal{CH}_m$ we can define a scalar product on $\mathcal{CH}_m$ as follows. For $f,g \in \mathcal{CH}_m$ let,
$$
\< f, g \>_{\mathcal{CH}_m} =  \int_{\R^d} d^d \hat{\xi} \,\, f(\hat{\xi}_1, , \dots, \hat{\xi_d}) \overline{g(\hat{\xi}_1,  \dots, \hat{\xi_d})}.
$$
This  gives rise to an isometric mapping $\hat{I}_d: \mathcal{CH}_m \to L^2(\Omega)$ via,
$$
\hat{I}_d(f) = d! \int_{\R^d} d^d\hat{B}_{\hat{\xi}} \, \,  f(\hat{\xi}_1, , \dots, \hat{\xi_d}), 
$$
see \cite{Diu-Tran,Major_mult_wiener_integrals,Dobrushin}. The next lemma gives us a way to relate $I_d$ and $\hat{I}_d$.

\begin{lemma}[\cite{Taqqu} Lemma 6.1]\label{lemma-spectral-representation}
	Let $h \in L^2(\R^d,\lambda)$ be a real valued symmetric function and denote its Fourier transform by $\hat{h}$, then, 
	$$I_d(h)=\int_{\R^d}d^d B_{\xi} \,\, h(\xi_1, \xi_2, \dots, \xi_{d}) = \int_{\R^d} d^d \hat{B}_{\hat{\xi}} \,\, \hat{h}(\hat{\xi}_1, \hat{\xi}_2, \dots, \hat{\xi}_d) = \hat{I}_d(\hat{h}), $$
	where $\hat{B}$ is a Gau{\ss}ian complex valued random spectral measure given as above and the second equality is to be understood in law.
\end{lemma}
This relation also holds in the multivariate setting.
\begin{lemma}\label{lemma-spectral-representation-multi}[ \cite{Taqqu_sums} Lemma A.2 ]
	Let,  for each $j = 1, \dots N$, a symmetric function $h^j \in L^2(\R^{d_j},\lambda)$ be given and denote its Fourier transform by $\hat{h}^j $,
	then, the following equivalence holds in law,
	$$ \left( I_{d_1}(h^1), \dots, I_{d_N}(h^N) \right) =  \left( \hat{I}_{d_1}(\hat{h}^1), \dots, \hat{I}_{d_N}(\hat{h}^N) \right),$$ 
	where $I_d$ and $\hat{I}_d$ denote the maps defined in sections \ref{section-preliminary} and \ref{subsubsection-spectral-measure} respectively.
\end{lemma}

\begin{lemma}\label{lemma-Fourier-transform}
	Given $g^{k,r}_t= f_t \tilde{\otimes}_{r_1} f_t \tilde{\otimes}_{r_2} f_t \dots \tilde{\otimes}_{r_k} f_t$ such that $\delta(k,r)=d$, which is a symmetric function in $L^2(\R^d, \lambda)$, then,
	$$
	\hat{g}^{k,r}_t(\hat{\xi}_1, \dots, \hat{\xi}_{d}) = \left( \f {\Gamma(H_0- \f 1 2)} {\sqrt{2 \pi} } \right)^d \sum_{\psi \in \mathcal{S}_r} C_3(\psi,r,H_0) \left( \hat{\phi}^{k,r,\psi}_t \left(\sum_{l=1}^{M_1(\psi)} \hat{\xi}_{1,l} , \dots , \sum_{l=1}^{M_k(\psi)} \hat{\xi}_{k,l}\right) \prod_{j=1}^k \prod_{l=1}^{M_j(\psi)} \hat{\xi}_{j,l}^{\f 1 2 - H_0}  C(\hat{\xi}_{j,l}) \right),
	$$ 
	where $\phi^{k,r,\psi}_t(s) = \1_{s \leq t} \prod_{j=1}^k x(t-s_j)  \prod_{q=1}^{j-1} \vert s_j - s_q \vert^{\beta_{j,q}(\psi) (2H_0-2)}$. Furthermore, the following equivalence holds in law,
	$$ 
	\int_{\R^d} d^dB_{\xi} g^{k,r}_t = \int_{\R^d} d^d \hat{B}_{\hat{\xi}}  \left( \f {\Gamma(H_0- \f 1 2)} {\sqrt{2 \pi} } \right)^d \sum_{\psi \in \mathcal{S}_r} C_3(\psi,r,H_0) \left( \hat{\phi}^{k,r,\psi} \left(\sum_{l=1}^{M_1(\psi)} \hat{\xi}_{1,l} , \dots , \sum_{l=1}^{M_k(\psi)} \hat{\xi}_{k,l}\right) \prod_{j=1}^k \prod_{l=1}^{M_j(\psi)} \hat{\xi}_{j,l}^{\f 1 2 - H_0}   \right).
	$$
\end{lemma}

\begin{proof}
	By section \ref{section-decomposition} 
	$$ g^{k,r}_t =  \sum_{\psi \in \mathcal{S}_r}  C_3(\psi,r,H_0) \int_{[-\infty,t]^k}  d^k s \prod_{j=1}^k
	x(t-s_j) 
	\prod_{l=1}^{M_j(\psi)} (s_j-\xi_{\psi(j,l)})^{H_0 - \f 3 2} 
	\prod_{q=1}^{j-1 }\vert s_j - s_q \vert^{\beta_{j,q}(\psi) (2H_0 -2 )}.$$
	As Fourier transforming commutes with symmetrizing we again restrict our computation to one $\psi$ as our computations are independent of the choice for $\psi$. Thus, we  suppress $\psi$ in our notation as well the constants $C_3$.
	Following Taqqu \cite{Taqqu}, care is needed as $\xi_{j,l}^{H_0-\f 3 2}$ does neither belong to $L^1$ nor $L^2$. Instead we first consider $g_t^{k,r,K}(\xi) = g^{k,r}_t(\xi) \1_{[-K,K]^d}(\xi)$.
	Substituting $ u_{j,l} = s_j - \xi_{j,l} $ leads us to,
	\begin{align*}
	&\hat{g}_t^{k,r,K} = \f{1}{(2\pi)^{\f d 2}} \int_{\R^d}d^d \xi \, \,  e^{i \sum_{j,l} \hat{\xi}_{j,l} \xi_{j,l}} \int_{[-\infty,t]^k} d^k s \prod_{j=1}^k x(t-s_j) 
	\prod_{l=1}^{M_j} (s_j-\xi_{j,l})_{+}^{H_0-\f 3 2} \1_{[0,t]}(s_j) 
	\prod_{q=1}^{j-1 }\vert s_j - s_q \vert^{\beta_{j,q} (2H_0-2)}\\
	&= \f{1}{(2\pi)^{\f d 2}} \int_{\R^d} d^d u \int_{[-\infty,t]^k} d^k s e^{i \sum_{j,l} \hat{\xi}_{j,l} (s_j-u_{j,l})}  \prod_{j=1}^k x(t-s_j) 
	\prod_{l=1}^{M_j} (u_{j,l})_{+}^{H_0-\f 3 2}  \1_{ [s_j - K, s_j +K]  }(u_{j,l})
	\prod_{q=1}^{j-1 }\vert s_j - s_q \vert^{\beta_{j,q} (2H_0-2)} \\
	&= \f{1}{(2\pi)^{\f d 2}}\int_{\R^d} d^d u \int_{[-\infty,t]^k} d^k s e^{- i \sum_{j,l} \hat{\xi}_{j,l} u_{j,l}}   \prod_{j=1}^k \prod_{l=1}^{M_j} (u_{j,l})_{+}^{H_0-\f 3 2} \1_{ [s_j - K, s_j +K]  }(u_{j,l})  e^{\sum_{j,l} \hat{\xi}_{j,l} s_j}  x(t-s_j) 
	\prod_{q=1}^{j-1 }\vert s_j - s_q \vert^{\beta_{j,q} (2H_0-2)}.
	\end{align*}
	
	Although we can not separate the $u$ and $s$ integrals as  $e^{- i \sum_{j,l} \hat{\xi}_{j,l} u_{j,l}}   \prod_{j=1}^k \prod_{l=1}^{M_j} (u_{j,l})_{+}^{H_0-\f 3 2} \1_{ [s_j - K, s_j +K]  }(u_{j,l})$ depends on $s$, as $K \to \infty$ this dependency  vanishes, thus, we go ahead and analyse its behaviour as $K \to \infty$. The $u_{j,l}$'s terms split so we can restrict ourselves to 
	$$
	Q_{\hat{\xi}_{j,l}}(a,b) = \f{1} {\sqrt{2\pi}} \int_a^b e^{i u_{j,l} \hat{\xi}_{j,l}} u_{j,l}^{H_0 - \f 3 2} du_{j,l},
	$$
	for $0 \leq a \leq b < \infty$.
	The following estimate was shown in \cite{Taqqu}, 
	
	$$
	\sup_{0 \leq a \leq b < \infty} \vert Q_{\hat{\xi}_{j,l}}(a,b) \vert\leq \f{1}{\sqrt{2\pi}} \left( \f {1} {H_0 -\f 1 2} + \f {2} { \vert \hat{\xi}_{j,l} \vert} \right).
	$$
	Thus,
	\begin{align*}
	&\vert \hat{g}^{k,r,K}_t(\hat{\xi}) \vert \leq \\
	&\int_{[-\infty,t]^k} d^k s \prod_{j=1}^k \vert x(t-s_j) 
	\prod_{q=1}^{j-1 }\vert s_j - s_q \vert^{\beta_{j,q} (2H_0-2)} \vert \prod_{j,l} B_{\hat{\xi}_{j,l}}(\max(0,s_j-K),\max(0,s_j+K)) \vert\\
	&\leq \int_{[-\infty,t]^k} d^k s \vert \phi_t^{k,r}(s) \vert \prod_{j,l} \f {1} {\sqrt{2\pi}} \left(\f{1} { H_0 - \f 1 2} + \f {2} { \hat{\xi}_{j,l}} \right).
	\end{align*}

	By arguments similar to the ones in Lemma \ref{lemma-variance-asymptotics} one can show $\phi^{k,r}_t \in L^1(\R^k,\lambda)$.
	We compute,
	\begin{align*}
	\Vert \phi^{k,r}_t \Vert_{L^1(\R^k, \lambda)} &= \int_{[-\infty,t]^k} d^k s \prod_{j=1}^k \vert x(t-s_j) \vert
	\vert s_k - s_j \vert^{\beta_{k,j} (2H_0-2)}.
	\end{align*}
	Looking at the integral with respect to $s_k$, and again pulling the kernels $x(t-s_j)$, with the right exponents, in the integral, we obtain, setting $ B_k = m - \sum_{j=1}^{k} \beta_{k,j},$
	\begin{align*}
	 &\int_{-\infty}^{t} ds_k \vert x(t-s_k)\vert^{ \f {B_k} {m} }   \prod_{j=1}^k  \vert x(t-s_k)\vert^{ \f {\beta_{k,j}} {m} } \vert  x(t-s_j)\vert^{ \f {\beta_{k,j}} {m} } 
	\vert s_k - s_j \vert^{\beta_{k,j} (2H_0-2)}\\
	&\leq \left( \int_{-\infty}^{t} ds_k \vert x(t-s_k)\vert^{ \f {B_k} {m} } \right) \prod_{j=1}^k \left( \int_{-\infty}^{t} ds_k  \vert x(t-s_k) x(t-s_j) \vert
	\vert s_k - s_j \vert^{ (2H_0-2)} \right)^{ \f {\beta_{k,j}} {m}  }\\
	&\leq  \Vert x \Vert_{L^1(\R,\lambda)}^{ \f {B_k} {m} }  \prod_{j=1}^k \left( \int_{-\infty}^{t} ds_k  \vert x(t-s_k) x(t-s_j) \vert
	\vert s_k - s_j \vert^{ (2H_0-2)} \right)^{ \f {\beta_{k,j}} {m}  }
	\end{align*}
	Iterating this procedure as in section \ref{section-variance-computation} we finally obtain
		\begin{align*}
	\Vert \phi^{k,r}_t \Vert_{L^1(\R^k, \lambda)} &\leq  \Vert x \Vert_{L^1(\R,\lambda)}^{ \sum_{j=1}^k \f {B_j} {m} }  \prod_{j,q=1}^k \left( \int_{-\infty}^{t} ds  \int_{-\infty}^{t} dr \vert x(t-s) x(t-r)
	\vert s - r \vert^{ (2H_0-2)} \right)^{ \f {\beta_{j,q}} {m}  }\\
	&\leq \Vert x \Vert_{L^1(\R,\lambda)}^{ \f {d} {m} }   \Vert x \Vert_{ \vert \mathcal{H} \vert}^{ \f {m-d} {m} }.
	\end{align*}

	Thus, $\hat{g}^{k,r,K}_t$ is finite and uniformly bounded with respect to $K$ as soon as all $\hat{\xi}_{j,l}$ are different from $0$.
	By the substitution $u \to u \vert \vert \hat{\xi}_{j,l} \vert$, we obtain
	\begin{align*}
	\hat{g}^{k,r}_t &= \lim_{K \to \infty} \hat{g}_t^{k,r,K} \\
	&= \int_{[-\infty,t]^k} d^ks  e^{i\sum_{j,l} \hat{\xi}_{j,l} s_j} \prod_{j=1}^k     x(t-s_j) 
	\prod_{q=1}^{j-1 }\vert s_j - s_q \vert^{\beta_{j,q} (2H_0-2)}    \prod_{j,l} \f{1} {\sqrt{2\pi}} \int_0^{\infty} e^{-i\hat{\xi}_{j,l} u} u^{H_0-\f 3 2} du\\
	&= \int_{[-\infty,t]^k} d^ks  e^{i\sum_{j,l} \hat{\xi}_{j,l} s_j} \prod_{j=1}^k     x(t-s_j) 
	\prod_{q=1}^{j-1 }\vert s_j - s_q \vert^{\beta_{j,q} (2H_0-2)} \prod_{j,l} \hat{\xi}_{j,l}^{\f 1 2 - H_0} \f{1} {\sqrt{2\pi}} \int_0^{\infty} e^{-iu \, sign(\hat{\xi}_{j,l})} u^{H_0 - \f 3 2} du\\
	&= \int_{[-\infty,t]^k} d^ks  e^{i\sum_{j,l} \hat{\xi}_{j,l} s_j} \prod_{j=1}^k     x(t-s_j) 
	\prod_{q=1}^{j-1 }\vert s_j - s_q \vert^{\beta_{j,q} (2H_0-2)} \prod_{j,l} \hat{\xi}_{j,l}^{\f 1 2 - H_0} \Gamma(H_0-\f 1 2) C(\hat{\xi}_{j,l})\\
	&= \hat{\phi}^{k,r}_t \left(\sum_{l=1}^{M_1} \hat{\xi}_{1,l} , \dots , \sum_{l=1}^{M_k} \hat{\xi}_{k,l}\right) \prod_{j,l} \hat{\xi}_{j,l}^{\f 1 2 - H_0} \Gamma(H_0-\f 1 2) C(\hat{\xi}_{j,l}).
	\end{align*} 
	For the identity $ \f{1} {\sqrt{2\pi}} \int_0^{\infty} e^{-iu \, sign(\hat{\xi}_{j,l})} u^{H_0 - \f 3 2} du =  \Gamma(H_0-\f 1 2) C(\hat{\xi}_{j,l})$ we refer to the appendix of \cite{Diu-Tran}. Furthermore,  $C(\hat{\xi}) = e^{i \f {\pi } {2} (H_0 - \f 1 2) }$ for $\hat{\xi}>0$, thus,  $C(-\hat{\xi}) = \overline{C(\hat{\xi})}$ and $\vert C(\hat{\xi}) \vert =1$ for $\hat{\xi} \not = 0$. Hence,  see \cite{Dobrushin},  $C(\hat{\xi}_{j,l}) dW_{\hat{\xi}_{j,l}} \sim dW_{\hat{\xi}_{j,l}}$. This concludes the proof.
\end{proof}

So, what did we gain from dealing with $\hat{g}^{k,r}$ instead of $g^{k,r}$? In the kernel representation for $g^{k,r}$ we see that the $s$ variables are in a way convoluted with the $\xi$ ones. As we saw in the graph picture above, when computing $L^2$ norms we could entangle them with our mince knife, however, for obtaining kernel convergence this method is too harsh. Nevertheless, as Fourier transforms change convolutions to multiplications this in a sense entangles the edges between the graphs from the ones within each graph.

\subsubsection{Kernel Convergence}

In the long range dependent case we set, for $T \in [0,1]$, 
$$   
\bar{g}^{k,r,\epsilon}_{T} = \epsilon^{H^*(d)} \int_0^{\f T \epsilon} g^{k,r}_t dt,
$$
where $d=\delta(k,r)$.
\begin{proposition}\label{proposition-joint-wiener-building-blocks-lrd}
	Given the process $ \left( I_{d_{1}}(\bar{g}^{k_1,r^1,\epsilon}_{T}) , \dots , I_{d_{L}}(\bar{g}^{k_L,r^L,\epsilon}_{T} ) \right)$, for which $d_j= \delta(k_j,r^j)$ such that for each $1 \leq j \leq L$, $H^*(d_j) > \f 1 2$, and $ T \in [0,1]$, then,
	$$ 
	\left(  I_{d_{1}}(\bar{g}^{k_1,r^1}_{T}) , \dots, I_{d_{L},k}(\bar{g}^{k_L,r^l}_{T} ) \right)  \to \left( \kappa_1 Z^{H^*(d_1),d_1}_T , \dots,  \kappa_L Z^{H^*(d_L),d_L}_T \right),
	$$
	where $\kappa_j = \lim_{\epsilon \to 0} \Vert I_{d_{j}}(\bar{g}^{k_j,r^j}_{1}) \Vert_{L^2}$  and  $\left( Z^{H^*(d_1),d_1}_T , \dots, Z^{H^*(d_L),d_L}_T \right)$ is a multidimensional Hermite process such that each component is defined via the same Wiener process, in the sense of finite dimensional distributions. 
\end{proposition}
\begin{proof}
 	By Lemma \ref{lemma-Fourier-transform} the Fourier transform of a building block is given by
	$$
	\hat{g}^{k,r}_t = \left( \f {\Gamma(H_0- \f 1 2)} {\sqrt{2 \pi} } \right)^d  \sum_{\psi \in \mathcal{S}_r} C_3(\psi,r,H_0)  \hat{\phi}^{k,r,\psi}_t \left(\sum_{l=1}^{M_1(\psi)} \hat{\xi}_{1,l} , \dots , \sum_{l=1}^{M_k(\psi)} \hat{\xi}_{k,l}\right) \prod_{j=1}^k \prod_{l=1}^{M_j(\psi)} \hat{\xi}_{j,l}^{\f 1 2 - H_0}  C(\hat{\xi}_{j,l}),
	$$
	where the constants $C(\hat{\xi}_{j,l})$ can be absorbed into the measures.
	We perform the following computations for one building block, however, the changes of variables can be done simultaneously for all components and hence preserve equivalence in law.
	As usual we drop constants and restrict ourselves to one choice of $\psi$ as our analysis is independent of this choice. Denoting by $\hat{\bar{g}}^{k,r,\epsilon}_T $ the Fourier transform of $\bar{g}^{k,r,\epsilon}_T $, we compute,
	\begin{align*}
	\hat{I}_d\left(\hat{\bar{g}}^{k,r,\epsilon}_T \right)=&\epsilon^{H^*(d)}  \int_0^{\f T \epsilon} dt \int_{\R^d} d^d \hat{B}_{\hat{\xi}} \, \,  \prod_{j=1}^{d} \vert \hat{\xi}_j \vert^{{\f 1 2 - H_0}} \hat{\phi}^{k,r}_t \left(\sum_{l=1}^{M_1} \hat{\xi}_{1,l} , \dots , \sum_{l=1}^{M_k} \hat{\xi}_{k,l}\right)\\
	&=\epsilon^{H^*(d)}  \int_0^{\f T \epsilon} dt \int_{\R^d} d^d \hat{B}_{\hat{\xi}} \, \,  \prod_{j=1}^{d} \vert \hat{\xi}_j \vert^{{\f 1 2 - H_0}}  \int_{[-\infty,t]^k} d^ks  e^{\sum_{j,l} \hat{\xi}_{j,l} s_j} \prod_{j=1}^k     x(t-s_j) 
	\prod_{q=1}^{j-1 }\vert s_j - s_q \vert^{\beta_{j,q} (2H_0-2)}.
	\end{align*}
	In the following we use the self-similarity of $\hat{B}$, namely, $\hat{B}(\epsilon \hat{\xi}) = \epsilon^{\f 1 2}  \hat{B}( \hat{\xi})$, cf. \cite{Dobrushin,Dobrushin-Major}.
	Applying the changes of variables  $ t \to \epsilon t$,  $\hat{\xi}_{j,l} \to \f 1 \epsilon \hat{\xi}_{j,l}$ and  $ s_j \to \f { t }  {\epsilon} - s_j$ we obtain,
	\begin{align*}
	\hat{I}_d\left(\hat{\bar{g}}^{k,r,\epsilon}_T \right)
	&= \epsilon^{H^*(d)-1}  \int_0^{ T } dt \int_{\R^d} d^d \hat{B}_{\hat{\xi}} \, \,  \prod_{j=1}^{d} \vert \hat{\xi}_j \vert^{{\f 1 2 - H_0}}  \int_{[-\infty,\f t \epsilon]^k} d^ks  e^{i \sum_{j,l} \hat{\xi}_{j,l} s_j} \prod_{j=1}^k     x\left(\f t \epsilon -s_j\right) 
	\prod_{q=1}^{j-1 }\vert s_j - s_q \vert^{\beta_{j,q} (2H_0-2)}\\
	&= \epsilon^{H^*(d)-1 - d H_0 +d}   \int_0^{ T } dt \int_{\R^d} d^d \hat{B}_{\hat{\xi}} \, \,  \prod_{j=1}^{d} \vert \hat{\xi}_j \vert^{{\f 1 2 - H_0}}  \int_{[-\infty,\f t \epsilon]^k} d^ks  e^{i \sum_{j,l} \epsilon \hat{\xi}_{j,l} s_j} \prod_{j=1}^k     x\left(\f t \epsilon -s_j\right) 
	\prod_{q=1}^{j-1 }\vert s_j - s_q \vert^{\beta_{j,q} (2H_0-2)}\\
	&=    \int_0^{ T } dt \int_{\R^d} d^d \hat{B}_{\hat{\xi}} \, \,  \prod_{j=1}^{d} \vert \hat{\xi}_j \vert^{{\f 1 2 - H_0}}  \int_{[-\infty,\f t \epsilon]^k} d^ks  e^{i \sum_{j,l} \epsilon \hat{\xi}_{j,l} ( \f { t }  {\epsilon} - s_j)} \prod_{j=1}^k     x\left(s_j\right) 
	\prod_{q=1}^{j-1 }\vert s_j - s_q \vert^{\beta_{j,q} (2H_0-2)}\\
	&=    \int_0^{ T } dt \int_{\R^d} d^d \hat{B}_{\hat{\xi}} \, \,  \prod_{j=1}^{d} \vert \hat{\xi}_j \vert^{{\f 1 2 - H_0}}  \int_{[-\infty,\f t \epsilon]^k} d^ks   e^{i \sum_{j,l} \hat{\xi}_{j,l} t} e^{-i \sum_{j,l} \epsilon \hat{\xi}_{j,l}  s_j} \prod_{j=1}^k     x\left(s_j\right) 
	\prod_{q=1}^{j-1 }\vert s_j - s_q \vert^{\beta_{j,q} (2H_0-2)}\\
	&= \hat{I}_d\left(\mathfrak{g}^{k,r,\epsilon}_T \right)
	\end{align*}
	Now, fix $ T \in [0,1]$. By the $L^2$ isometry property obtained in section \ref{subsubsection-spectral-measure} we can also work with the kernel in order to prove $L^2(\Omega)$ convergence. 
	Hence, by dominated convergence we obtain the pointwise result
	\begin{align*}
	&\int_0^{ T } dt \int_{\R^d} d^d {\hat{\xi}} \, \,  \prod_{j=1}^{d} \vert \hat{\xi}_j \vert^{{\f 1 2 - H_0}}  \int_{[-\infty,\f t \epsilon]^k} d^ks   e^{i \sum_{j,l} \hat{\xi}_{j,l} t} e^{-i \sum_{j,l} \epsilon \hat{\xi}_{j,l}  s_j} \prod_{j=1}^k     x\left(s_j\right) 
	\prod_{q=1}^{j-1 }\vert s_j - s_q \vert^{\beta_{j,q} (2H_0-2)}\\
	&\to  \int_{\R^d}  d^d \hat{\xi}  \f { e^{i t \sum_{j=1}^m \hat{\xi}_j} -1 } { i \sum_{j=1}^m \hat{\xi}_j} \prod_{j=1}^m \vert \hat{\xi}_j \vert^{H_0 - \f 1 2}   \int_{\R^k}  d^ks   \prod_{j=1}^k     x\left(s_j\right) 
	\prod_{q=1}^{j-1 }\vert s_j - s_q \vert^{\beta_{j,q} (2H_0-2)} ,
	\end{align*}
	which is up to the constant $\int_{\R^k}  d^ks   \prod_{j=1}^k     x\left(s_j\right) 
	\prod_{q=1}^{j-1 }\vert s_j - s_q \vert^{\beta_{j,q} (2H_0-2)} $ the spectral representation of $Z^{H^*(d),d}_T$, see Equation \ref{spectral-representation-Hermiteprocesses}.
	We now show  that this sequence is Cauchy in $L^2(\Omega)$. For $\epsilon_2 < \epsilon_1$,
	\begin{align*}
	&\Big \Vert \epsilon_1^{H^*(d)}  \int_0^{\f {T} {\epsilon_1}} \hat{I}_d\left(\hat{g}^{k,r}_{t}\right) dt - \epsilon_2^{H^*(d)} \int_0^{\f {T} {\epsilon_2}} \hat{I}_d\left(\hat{g}^{k,r}_{t}\right) dt  \Big \Vert_{L^2(\Omega)} \\
	&\leq  \int_0^{ T } dt \int_{\R^d} d^d \hat{\xi}  \Bigg ( \prod_{j=1}^{d} \vert \hat{\xi}_j \vert^{{\f 1 2 - H_0}}  \int_{ [-\infty,\f {t} {\epsilon_1}]^k } d^ks   e^{i \sum_{j,l} \hat{\xi}_{j,l} t} e^{-i \sum_{j,l} (\epsilon_1 - \epsilon_2) \hat{\xi}_{j,l}  s_j} \prod_{j=1}^k     x\left(s_j\right) 
	\prod_{q=1}^{j-1 }\vert s_j - s_q \vert^{\beta_{j,q} (2H_0-2)}\\  
	&+  \prod_{j=1}^{d} \vert \hat{\xi}_j \vert^{{\f 1 2 - H_0}}  \int_{[-\infty,\f {t} {\epsilon_2}]^k \setminus [-\infty,\f {t} {\epsilon_1}]^k } d^ks   e^{i \sum_{j,l} \hat{\xi}_{j,l} t} e^{i \sum_{j,l}   \epsilon_2 \hat{\xi}_{j,l}  s_j} \prod_{j=1}^k     x\left(s_j\right) 
	\prod_{q=1}^{j-1 }\vert s_j - s_q \vert^{\beta_{j,q} (2H_0-2)} \Bigg )^2,
	\end{align*}
	which, by dominated convergence, goes to $0$ as $ \epsilon_1, \epsilon_2 \to 0$.
	Hence, for each permutation and $T \in [0,1]$ we obtain convergence in $L^2(\Omega)$. 
	To conclude convergence in finite dimensional distributions we want to argue by the Cramer-Wold Theorem \ref{Cramer-Wold-theorem}.
	Thus given finitely many $T_{i,j} \in [0,1]$, $0 \leq i \leq Q_j$, we need to investigate the vector 
	$$\left(  I_{d_{1}}(\bar{g}^{k_1,r^1}_{T_{1,1}}) , \dots , I_{d_{1}}(\bar{g}^{k_1,r^1}_{T_{Q_1,1}} ) \dots, I_{d_{L}}(\bar{g}^{k_L,r^l}_{T_{Q_L,L}} ) \right).$$
	As mentioned above the changes of variables $ t \to \epsilon t$,  $\hat{\xi}_{j,l} \to \epsilon \hat{\xi}_{j,l}$ and  $ s_j \to \f { t }  {\epsilon} - s_j$ can be performed simultaneously and we obtain the following equivalence in law,
	\begin{align*}
\left(  I_{d_{1}}(\bar{g}^{k_1,r^1}_{T_{1,1}}) , \dots , I_{d_{1}}(\bar{g}^{k_1,r^1}_{T_{Q_1,1}} ) \dots, I_{d_{L}}(\bar{g}^{k_L,r^l}_{T_{Q_L,L}} ) \right) = \left(  I_{d_{1}}(\bar{\mathfrak{g}}^{k_1,r^1}_{T_{1,1}}) , \dots , I_{d_{1}}(\bar{\mathfrak{g}}^{k_1,r^1}_{T_{Q_1,1}} ) \dots, I_{d_{L}}(\bar{\mathfrak{g}}^{k_L,r^l}_{T_{Q_L,L}} ) \right).
	\end{align*}
	By the above each component of the later converges in $L^2(\Omega)$ to the sepctral representation of the proclaimed limit. As $L^2(\Omega)$ convergence in each component implies $L^2(\Omega)$ convergence of the whole vector this concludes the proof. 
\end{proof}

\section{General Functions}
In the previous section we established joint convergence in finite dimensional distributions for our building blocks. As $G(y_s)$ may consist of infinitely many such objects we first prove so called reduction theorems in both the short and long range dependent setup, cf. \cite{Taqqu,Taqqu_sums}. In the Gau{\ss}ian setup it suffices to look at finitely many Hermite polynomials in  the SRD regime and only at the lowest rank one in the LRD case. In our case  each term $(y_s)^k$ could give us a contribution in each chaos up to order $km$, hence, we do not just make a cut-off in the chaos rank, but also in the ranks of the polynomials. However similar to the Gau{\ss}ian case it suffices to consider only the lowest chaos rank contributions in the long range dependent case. Furthermore, we use the decay assumption imposed on the coefficients $c_k$ to ensure that our estimates from sections \ref{section-variance-computation} and \ref{subsection-tightness-Hoelder} can be carried over.

\subsection{Short range dependent case}

\begin{definition}
	Given $G(X) = \sum_{k =0}^{\infty} c_k X^k$ such that $\vert c_k \vert \lesssim \f 1 {k!}$ we denote by $G_{M}(y_t)$ the projection of $\sum_{k=0}^M c_k \left(y_t\right)^k$ onto the first $M$ Wiener chaoses minus the higher polynomial contributions in the $0^{th}$ chaos. Thus, $G_M(y_t) = \sum_{k=0}^{M} c_k \sum_{d=0}^M I_d(h^{d,k}_t) + \sum_{k=M+1}^{\infty}  c_k I_0(h^{0,k}_t). $  
\end{definition}
\begin{remark}
The term $\sum_{k=M+1}^{\infty}  I_0(h^{0,k}_t)$ ensures that $\E\left[  G_M(y_s) \right] = 0$, hence, also $\E \left[ G(y_s) - G_M(y_s) \right]=0$, for centred functions $G$. 
\end{remark}
\begin{lemma}\label{lemma-conditions-theorem-reduction-wiener}
Given $G(X) = \sum_{k=0}^{\infty} c_{k} X^k$ such that $G$ has chaos rank $w \geq 1 $ with respect to $y_t$, $\vert c_k \vert \lesssim \f {1} {k!}$ and $H^*(w) \in (-\infty,\f 1 2 )$, then, 
$$
\lim_{M \to \infty} \limsup_{\epsilon \to 0 } \E \left[ \left( \sqrt{\epsilon} \int_0^{\f T \epsilon} G(y_t) -G_M(y_t) dt \right)^2 \right] = 0.
$$
\end{lemma}
\begin{proof}
\begin{align*}
&\E \left[ \left( \sqrt{\epsilon} \int_0^{\f T \epsilon} G(y_t) -G_M(y_t) dt \right)^2 \right]\\
&=   \epsilon  \int_0^{\f T \epsilon} \int_0^{\f T \epsilon} dt dt'  \E \left[ \left( G(y_t) -G_M(y_t) \right) \left(G(y_{t'}) -G_M(y_{t'}) \right) \right]   \\
&= \epsilon \int_0^{\f T \epsilon} \int_0^{\f T \epsilon} dt dt' \E \Big [ \left( \sum_{k=0}^{\infty} \sum_{d=M+1}^{\infty} c_k I_d\left(h^{d,k}_t \right) + \sum_{k=M+1}^{\infty} \sum_{d=1}^{M} c_k I_d\left(h^{d,k}_t \right) \right) \\ &\left( \sum_{k'=0}^{\infty} \sum_{d'=M+1}^{\infty} c_{k'} I_{d'}\left(h^{d',k'}_{t'} \right) + \sum_{k'=M+1}^{\infty} \sum_{d'=0}^{M} c_{k'} I_{d'}\left(h^{d',k'}_{t'} \right) \right)  \Big ]  \\
&=  \sum_{d=M+1}^{\infty}  \sum_{k,k'=0}^{\infty}  c_k c_{k'} \epsilon \int_0^{\f T \epsilon} \int_0^{\f T \epsilon}  dt dt' \E \left[I_d(h^{d,k}_t)   I_d(h^{d,k'}_{t'}) \right] + \sum_{d=1}^{M}  \sum_{k,k'=M+1}^{\infty}  c_k c_{k'} \epsilon \int_0^{\f T \epsilon} \int_0^{\f T \epsilon} dt dt' \E \left[I_d(h^{d,k}_t)   I_d(h^{d,k'}_{t'}) \right]  \\
&\lesssim \sum_{d=M+1}^{\infty}  \sum_{k,k'=0}^{\infty}  c_k c_{k'} C_4(k,k',m,d)  +  \sum_{d=1}^{M}  \sum_{k,k'=M+1}^{\infty}  c_k c_{k'} C_4(k,k',m,d) .
\end{align*}
The first sum represents the part beloning to chaoses of rank bigger than $M$ and the second one the parts in a low order chaos from high order polynomials. Furthermore, $C_4(k,k',m,d) = \f {\sqrt{(km)!(k'm)!} (m!)^{(k+k')} (k+k')^m ((k+k')m)^2 m^{k+k'} \mathfrak{L}^{(k+k')m}} {d^2}$ satisfies,
\begin{align*}
&\sum_{d=M+1}^{\infty} \sum_{k,k'=0}^{\infty}   c_k c_{k'} C_4(k,k',m,d) + \sum_{d=1}^{M} \sum_{k,k'=M+1}^{\infty} c_k c_{k'} C_4(k,k',m,d) \\
& \lesssim \sum_{d=M+1}^{\infty} \f 1 {d^2} + \sum_{k,k'=M+1}^{\infty} \f {(m!)^{(k+k')} (k+k')^m ((k+k')m)^2 m^{k+k'} \mathfrak{L}^{(k+k')m}} {k! k'!} \\
& \xrightarrow{M \to \infty } 0 ,
\end{align*}
proving the claim.	
\end{proof}

\begin{lemma}\label{Reduktion-SRD}
 Let, for each $j =1, \dots, N$,  a function of the form $G^j(X) = \sum_{k=0}^{\infty} c_{j,k} X^k$ such that  $ \vert c_{j,k} \vert \lesssim \f {1}{k!}$ be given. If for every $M \in \N $,  
$$\left( \sqrt{\epsilon} \int_0^{\f T \epsilon} G_{M}^1(y_t)dt, \dots,  \sqrt{\epsilon} \int_0^{\f T \epsilon} G_{M}^N(y_t) dt \right),
$$
converges in finite dimensional distribution to a Wiener process $W^{1}_{M,T}=(W^1_{M,T}, \dots, W^N_{M,T})$ with covariance structure $\Lambda_M^{j,l}$ and $\lim_{M \to \infty}  \Lambda_M^{j,l} = \Lambda^{j,l}$, then,  
$$
\left( \sqrt{\epsilon} \int_0^{\f T \epsilon} G^1(y_t)dt, \dots,  \sqrt{\epsilon} \int_0^{\f T \epsilon} G^N(y_t) dt \right) 
$$
converges to a Wiener process $W_T$ with covariance structure $\Lambda^{j,l}$.
\end{lemma}
\begin{proof}
By Lemma \ref{lemma-conditions-theorem-reduction-wiener}the condition on the $L^2(\Omega)$ norm imposed in Theorem \ref{billingsley-3.2} is satisfied by the processes  $\left( \sqrt{\epsilon} \int_0^{\f T \epsilon} G_{M}^1(y_t)dt, \dots,  \sqrt{\epsilon} \int_0^{\f T \epsilon} G_{M}^N(y_t) dt \right)
$ and
$\left( \sqrt{\epsilon} \int_0^{\f T \epsilon} G^1(y_t)dt, \dots,  \sqrt{\epsilon} \int_0^{\f T \epsilon} G^N(y_t) dt \right)$. Thus, an application of Theorem \ref{billingsley-3.2} and the Cramer-Wold Theorem \ref{Cramer-Wold-theorem} concludes the proof.
\end{proof}

\begin{proposition}\label{reduziertes-SRD-Resultat}
Given a collection of functions $G^{j}$, $j=1, \dots, N$,  where $G^j = \sum_{k=0}^{\infty} c_{j,k} X^k$ such that $\vert c_{j,k} \vert \lesssim k!$,  $ G_M^j(y_s)$ as above and set $\bar{G}_{M,T}^{j,\epsilon} = \sqrt{\epsilon} \int_0^{\f T \epsilon} G_M^j(y_s) ds$. Then, for every $M \in \N$ and finite collection of times
$T_{i,j} \in [0,1]$ the vector $( \bar{G}^{j,\epsilon}_{M,T_{i,j}} ),$ where $0 \leq i \leq Q$ and $1 \leq j \leq N$, converges jointly to a multivariate normal distribution $(W^j_{M,T_{i,j}})$ with covariance structure
$$ \E \left[  W^{j}_{M,T_{i,j}} W^{j'}_{M,T_{i',j'}} \right] = \lim_{\epsilon \to 0} \E\left[ \bar{G}_{M,T_{i,j}}^{j,\epsilon} \bar{G}_{M,T_{i',j'}}^{j',\epsilon}   \right] = 2 (T_{i,j} \wedge T_{i',j'}) \sum_{k,k'=0}^{M} \sum_{d=0}^{M} \int_0^{\infty} \E \left[ I_d(h^{d,k}_s) I_d(h^{d,k'}_0)  \right] ds.
$$
\end{proposition}

\begin{proof}
As we now deal with finitely many terms $h^{d,k}$, there are infinitely many terms in the $0^{th}$ chaos, however by assumption they sum up to $0$,  we can collect all buildings blocks $g^{k_q,r^q}$ such that $k_q \leq M$ and  $0<\delta(k_q,r^q) \leq M$.
By Proposition \ref{proposition-joint-wiener-building-blocks-srd} the vector 
$( \bar{g}^{k_q,r^q}_{T_{i,j}})$ converges jointly to a multivariate normal distribution $(W^{q}_{T_{i,j}})$ with covariance 
$$\E \left[ W^{q}_{T_{i,j}} W^{q'}_{T_{i',j'}} \right] = 2 \left( T_{i,j} \wedge T_{i'.j'} \right) \int_0^{\infty} \E \left[ g^{k_q,r^q}_s g^{k_{q'},r^{q'}}_0 \right] ds.$$
As summation is a continuous operation, and 
$$
\bar{G}^j_{M,T_{i,j}} = \sum_{k=0}^M c_k \sum_{d=0}^M \sum_{r: \delta(k,r) = d} I_d(\bar{g}^{k,r}_{T_{i,j}}) - \E \left[ \sum_{k=0}^M c_k \sum_{d=0}^M \sum_{r: \delta(k,r) = d} I_d(\bar{g}^{k,r}_{T_{i,j}}) \right] 
$$
$(G^j_{M,T_{i,j}})_{\{ 0 \leq i \leq Q, 1 \leq j \leq N \}}$ also converges to a multivariate Gau{\ss}ian with the proclaimed covariances.
\end{proof}

\begin{proposition}
Fix $H \in ( \f 1 2, 1)$, $m \in \N$, a kernel $x$ satisfying assumptions \ref{assumption-decay-correlation-kernel} and set $y_t = \int_{-\infty}^t x(t-s) dZ^{H,m}_s$. Let, for each $j =1, \dots, N$,  a function of the form $G^j(X) = \sum_{k=0}^{\infty} c_{j,k} X^k$ such that $G^j$ has chaos rank $w^j$ with respect to $y_t$ and $\vert c_{j,k} \vert \lesssim \f {1} {k!}$ be given. Assume further that for each $j$,  $H^*(w^j) < \f 1 2 $ . For $T \in [0,1]$ set,
$$
\bar{G}^{j,\epsilon}_T = \sqrt{\epsilon} \int_0^{\f T \epsilon} G^j(y_t) dt.
$$
Then, the vector
$$
( \bar{G}^{1,\epsilon}_T , \dots, \bar{G}^{N,\epsilon}_T), 
$$
converges as $\epsilon \to 0$  weakly in $\C^{\gamma}([0,1],\R^N)$, for $\gamma \in (0, \f 1 2)$ to a multivariate Wiener process 
$$
W_T=\left(W^1_T, \dots , W^N_T\right)
$$ with covariance structure, for $T,S \in [0,1]$,
$$
\E\left[ W^j_T W^l_S \right] = 2 ( T \wedge S) \int_0^{\infty} \E \left[ G^{j}(y_s) G^{l}(y_0) \right] ds. 
$$
\end{proposition}
\begin{proof}
Combining Lemma \ref{Reduktion-SRD}, Proposition \ref{reduziertes-SRD-Resultat} and the computation 
\begin{align*}
\lim_{M \to \infty} \E\left[ W^j_{M,T} W^l_{M,S} \right] &= \lim_{M\to \infty} 2 (T \wedge S) \sum_{k,k'=0}^{M} \sum_{d=0}^{M} \int_0^{\infty} \E \left[ I_d(h^{d,k}_s) I_d(h^{d,k'}_0)  \right] ds\\
&= 2 ( T \wedge S) \int_0^{\infty} \E \left[ G^{j}(y_s) G^{l}(y_0) \right] ds\\
&=  \E\left[ W^j_T W^l_S \right]
\end{align*}
gives us convergence in finite dimensional distributions.
Furthermore, by Proposition \ref{proposition-hoelder} we have for each $j$ and $p>2$,
$$ \Vert \bar{G}^{j,\epsilon}_T - \bar{G}^{j,\epsilon}_S \Vert_{L^p(\Omega)} \lesssim \sqrt{\vert T-S \vert}.
$$
Thus, by an application of Kolmogorv's Theorem each $\bar{G}^{j,\epsilon}$ is tight in $\C^{\gamma}([0,1],\R^N)$ for $\gamma \in (0, \f 1 2)$. As tightness in $\C^{\gamma}([0,1],\R^N)$ is equivalent to tightness in each component this concludes the proof.
\end{proof}

\subsection{Long range dependent case}
In the long range dependent case the reduction is in fact easier as we only need to consider the lowest order chaos.  We denote by $G_{w,M}(y_s)$ the projections of the first $M$ polynomials of $G(y_s)$ onto the lowest order chaos, thus $G_{w,M}(y_s)=\sum_{k=0}^M c_k  I_w(h^{w,k})  $, in case $G$ has chaos rank $w$.
\begin{lemma}\label{lemma-reduction-Hermite-case-LRD}
	Given $G(X) = \sum_{k =0}^{\infty} c_k X^k$ such that $\vert c_k \vert \lesssim \f 1 {k!}$, $G$ has chaos rank $w \geq 1$  with respect to $y_s$, where $ H^*(w) > \f 1 2$, and $T \in [0,1]$, then, 
	$$ \lim_{M \to \infty} \limsup_{\epsilon \to 0} \Vert \epsilon^{H^*(w)}\left( \int_0^{\f T \epsilon} G(y_t) -G_{w,M}(y_t) dt \right) \Vert_{L^2(\Omega)} = 0. $$
\end{lemma}
\begin{proof}
	Using the estimates obtained in Lemma \ref{wachstum-konstanten}, we compute, similar to Lemma \ref{Reduktion-SRD}, 
	\begin{align*}
	&\E \left[ \left( \epsilon^{H^*(w)} \int_0^{\f T \epsilon} G(y_t) -G_{w,M}(y_t) dt \right)^2 \right]\\
	&=    \epsilon^{H^*(w)}  \int_{[0, \f T \epsilon]^2} \E \left[ \left( G(y_t) -G_{w,M}(y_t) \right) \left(G(y_{t'}) -G_{w,M}(y_{t'}) \right) \right] dt dt'   \\
	&=  \epsilon^{H^*(w)} \int_{[0, \f T \epsilon]^2} \E \Big [ \left( \sum_{k=0}^{\infty} \sum_{d=w+1}^{\infty} c_k I_d\left(h^{d,k}_t \right) + \sum_{k=M+1}^{\infty} c_k I_w\left(h^{w,k}_t \right) \right) \\ &\left( \sum_{k'=0}^{\infty} \sum_{d'=w+1}^{\infty} c_{k'} I_{d'}\left(h^{d',k'}_{t'} \right) + \sum_{k'=M+1}^{\infty}  c_{k'} I_{w}\left(h^{w,k'}_{t'} \right) \right)  \Big ] dt dt' \\
	&=  \sum_{d=w+1}^{\infty}  \sum_{k,k'=0}^{\infty}  c_k c_{k'}  \epsilon^{H^*(w)} \int_{[0, \f T \epsilon]^2} \E \left[I_d(h^{d,k}_t)   I_d(h^{d,k'}_{t'}) \right] dt dt' +   \sum_{k,k'=M+1}^{\infty}  c_k c_{k'}  \epsilon^{H^*(w)} \int_{[0, \f T \epsilon]^2} \E \left[I_d(h^{w,k}_t)   I_d(h^{w,k'}_{t'}) \right] dt dt' \\
	&\lesssim  \epsilon^{H^*(w)}o(\epsilon^{H^*(w)}) \sum_{d=w+1}^{\infty}  \sum_{k,k'=0}^{\infty}  c_k c_{k'} C_4(k,k',m,d)  +  \sum_{k,k'=M+1}^{\infty}  c_k c_{k'} C_4(k,k',m,w)  .
	\end{align*}
	Furthermore, $C_4(k,k',m,d) = \f {(m!)^{(k+k')} (k+k')^m ((k+k')m)^2 m^{k+k'} \mathfrak{L}^{(k+k')m}} {d^2}$ satisfies,
	\begin{align*}
	&\sum_{d=w+1}^{\infty} \sum_{k,k'=0}^{\infty}   c_k c_{k'} C_4(k,k',m,d) +  \sum_{k,k'=M+1}^{\infty} c_k c_{k'} C_4(k,k',m,w) \\
	& \lesssim \sum_{d=w+1}^{\infty} \f 1 {d^2} + \sum_{k,k'=M+1}^{\infty} \f {(m!)^{(k+k')} (k+k')^m ((k+k')m)^2 m^{k+k'} \mathfrak{L}^{(k+k')m}} {k! k'!},
	\end{align*}
	hence, $ \epsilon^{H^*(w)}o(\epsilon^{H^*(w)}) \sum_{d=w+1}^{\infty}  \sum_{k,k'=0}^{\infty}  c_k c_{k'} C_4(k,k',m,d) \to 0$ as $\epsilon \to 0$ and $ \sum_{k,k'=M+1}^{\infty}  c_k c_{k'} C_4(k,k',m,w) \to 0$ as $M \to \infty$, proving the claim.
\end{proof}
\begin{proposition}\label{reduziertes-LRD-Resultat}
	Given a collection of functions $G^{j}$, $j \in \{ 1, \dots, N \} $,  where $G^j = \sum_{k=0}^{\infty} c_{j,k} X^k$ such that $\vert c_{j,k} \vert \lesssim k!$ with chaos rank $w_j$,  $ G_{w_j,M}^j(y_s)$ as above and set $\bar{G}_{w_j,M,T}^{j,\epsilon} = \epsilon^{H^*(w_j)} \int_0^{\f T \epsilon} G_{w,M}^j(y_s) ds$. Then, for every $M \in \N$ and finite collection of times
	$T_{i,j} \in [0,1]$ the vector $( \bar{G}^{j,\epsilon}_{w_j,M,T_{i,j}} ),$ where $1 \leq i \leq Q$ and $1 \leq j \leq N$, converges jointly to the marginals of a multivariate Hermite process $( \kappa_{j,M} Z^{H^*(w_j),w_j}_{M,T_{i,j}})$, where $\kappa_{j,M} = \lim_{\epsilon \to 0} \Vert   \bar{G}^{j,\epsilon}_{w_j,M,1} \Vert_{L^2(\Omega)}$.
	In particular, 
	$$ \left( \bar{G}_{w_1,M,T}^{1,\epsilon},\dots , \bar{G}_{w_N,M,T}^{N,\epsilon} \right),$$
	converges in the sense of finite dimensional distributions to a multivariate Hermite process 
	$$\left( \kappa_{1,M} Z^{H^*(w_1),w_1}_{M,T} ,\dots , \kappa_{N,M} Z^{H^*(w_N),w_N}_{M,T} \right),$$
	where each component is defined via the same Wiener process.
\end{proposition}

\begin{proof}	
	As we now deal with finitely many terms $h^{d,k}$ we can view all buildings blocks $g^{k_q,r^q}$ such that $k_q \leq M$ and  $\delta(k_q,r^q) = w$.
	By Proposition \ref{proposition-joint-wiener-building-blocks-lrd} the vector 
	$( \bar{g}^{k_q,r^q}_{T_{i,j}})$ converges jointly to a multivariate Hermite distribution $(Z^{q,i,j}_{w,M,T_{i,j}}),$ where each component is defined via the same Wiener process.
	As summation is a continuous operation, and 
	$$
	\bar{G}^j_{w,M,T_{i,j}} = \sum_{k=0}^M  \sum_{r: \delta(k,r) = w} I_d(\bar{g}^{k,r}_{T_{i,j}}),
	$$
	also $(G^j_{M,T_{i,j}})$ converges to a multivariate Hermite distribution with the proclaimed covariances.
\end{proof}
\begin{proposition}
Fix $H \in ( \f 1 2, 1)$, $m \in \N$, a kernel $x$ satisfying assumptions \ref{assumption-decay-correlation-kernel} and set $y_t = \int_{-\infty}^t x(t-s) dZ^{H,m}_s$. Let, for each $j \in \{ 1, \dots, N \}$,  a function of the form $G^j(X) = \sum_{k=0}^{\infty} c_{j,k} X^k$ such that $G^j$ has chaos rank $w^j$ with respect to $y_t$ and $\vert c_{j,k} \vert \lesssim \f {1} {k!}$ be given. Assume further that for each $j$,  $H^*(w^j) > \f 1 2 $. For $T \in [0,1]$  set,
	$$
	\bar{G}^{j,\epsilon}_T = \sqrt{\epsilon} \int_0^{\f T \epsilon} G^j(y_t) dt.
	$$
	Then, the vector,
	$$
	( \bar{G}^{1,\epsilon}_T , \dots, \bar{G}^{N,\epsilon}_T), 
	$$
	converges as $\epsilon \to 0$ weakly in $\C^{\gamma}([0,1],\R^N)$, for $\gamma \in (0, \min_{j=1, \dots, N} H^*(w_j))$  to a multivariate Hermite process $$\left(\kappa_1 Z^{H^*(w_1),w_1}_T, \dots,  \kappa_N Z^{H^*(w_N),w_N}_T \right),$$
	where $ \kappa_j = \lim_{\epsilon \to 0} \Vert  \bar{G}^{j,\epsilon}_1 \Vert_{L^2(\Omega}$.
\end{proposition}
\begin{proof}
	Setting $\err_{w,M} = ( \bar{G}^{1,\epsilon}_T , \dots, \bar{G}^{N,\epsilon}_T) -  ( \bar{G}_{w,M,T}^{1,\epsilon} , \dots, \bar{G}_{w,M,T}^{N,\epsilon})$, we obtain
	\begin{align*}
	( \bar{G}^{1,\epsilon}_T , \dots, \bar{G}^{N,\epsilon}_T)= ( \bar{G}_{w,M,T}^{1,\epsilon} , \dots, \bar{G}_{w,M,T}^{N,\epsilon}) + \err_{w,M},
	\end{align*}
	where by Lemma \ref{lemma-reduction-Hermite-case-LRD} 
	$$ 
	\lim_{M \to \infty} \limsup_{\epsilon \to 0} \Vert \err_{w,M} \Vert_{L^2} = 0.
	$$
	Moreover, by Proposition \ref{reduziertes-LRD-Resultat} $( \bar{G}_{w,M,T}^{1,\epsilon} , \dots, \bar{G}_{w,M,T}^{N,\epsilon})$ converges in finite dimensional distributions to 
	$$
	(\kappa_{1,M} Z^{H^*(w_1),w_1}_{w_1,M,T}, \dots,  \kappa_{N,M} Z^{H^*(w_N),w_N}_{w_N,M,T}).
	$$
	As $\kappa_{j,M} \to \kappa_{j}$ as $M \to \infty$ and all our Hermite processes are defined via Wiener integrals over the same Wiener process we may apply Theorem \ref{billingsley-3.2} to conclude the first part of the proof.\\
	Concerning weak convergence in H\"older spaces, by Proposition \ref{proposition-hoelder} we have for each $j$ and $p>2$,
	$$ \Vert \bar{G}^{j,\epsilon}_T - \bar{G}^{j,\epsilon}_S \Vert_{L^p(\Omega)} \lesssim \vert T-S \vert^{H^*(w_j)}.
	$$
	Thus, by an application of Kolmogorv's Theorem each $\bar{G}^{j,\epsilon}_T$ is tight in $\C^{\gamma}([0,1],\R^N)$ for $\gamma \in (0, \min_{ \{ j=1, \dots , N \} } H^*(w_j)$. As tightness in $\C^{\gamma}([0,1],\R^N)$ is equivalent to tightness in each component this concludes the proof.
\end{proof}

\section{Mixed}
In this section we put together the Propositions \ref{proposition-joint-wiener-building-blocks-srd} and \ref{proposition-joint-wiener-building-blocks-lrd} into Theorem \ref{theorem-mixed}. To do so we rely on an asymptotic independence argument proven in \cite{Nourdin-Nualart-Peccati}.

\textbf{ Proof of Theorem \ref{theorem-mixed}}.

\begin{proof}
By the Lemmas  \ref{lemma-reduction-Hermite-case-LRD}, \ref{Reduktion-SRD},  Theorem \ref{billingsley-3.2} and arguments as in the proofs of the Propositions \ref{proposition-joint-wiener-building-blocks-srd} and \ref{proposition-joint-wiener-building-blocks-lrd}   it is sufficient to prove the claim for $\bar{G}^{j,\epsilon}_T$ replaced by $\bar{G}^{j,\epsilon}_{M,T}$ for arbitrary $M$ in case $j \leq n$ and $ \bar{G}^{j,\epsilon}_{w,M,T}$ in case $j > n$. Hence, we deal with finitely many terms given as iterated Wiener integrals. Furthermore for $j \leq n$ the functional  $\bar{G}^{j,\epsilon}_{M,T}$ is constructed by objects of the form $\bar{g}^{k,r}_T$ for which $\delta(k,r) = d$ such that $H^*(d)< \f 1 2$. In particular they converge to a Wiener process, thus, given two such terms, by the Fourth Moment Theorem \ref{fourth-moment-theorem}
$$
\Vert \bar{g}^{k_1,r^1,\epsilon}_{T} \otimes_r \bar{g}^{k_1,r^1,\epsilon}_{T} \Vert_{L^2(\R^{2d_1-2r},\lambda)} \to 0, \qquad  r =1 ,\dots , \min (k_1-1),
$$ 
and similarly for $\bar{g}^{k_2,r^2}_T$.

Applying Cauchy-Schwarz we obtain
for $r =1, \dots, \min (k_1-1,k_2-1)$, 
$$\left \Vert \hat  \bar{g}^{k_1,r^1,\epsilon}_{T} \otimes_r \bar{g}^{k_2,r^2,\epsilon}_{S}\right \Vert_{L^2(\R^{d_1+d_2-2r},\lambda)}
\leq \left \Vert \bar{g}^{k_1,r^1,\epsilon}_{T} \otimes_r \bar{g}^{k_1,r^1,\epsilon}_{T} \right \Vert_{L^2(\R^{d_1-r},\lambda)}  \left \Vert \bar{g}^{k_2,r^2,\epsilon}_{S} \otimes_r \bar{g}^{k_2,r^2,\epsilon}_{S} \right\Vert_{L^2(\R^{d_2-r},\lambda)}
\to 0,
$$
for all $S,T \in [0,1]$.
Now, an application of Proposition \ref{proposition-spit-independence} and the Cramer Wold Theorem \ref{Cramer-Wold-theorem} gives us the convergence in finite dimensional distributions of the vector
$$	( \bar{G}^{1,\epsilon}_T , \dots, \bar{G}^{N,\epsilon}_T).$$
Therefore, the moments bounds obtained in Lemma \ref{lemma-kolmogorv-bound} and an application of Kolmogorv's Theorem proof convergence in the proclaimed H\"older spaces.

To show that the Wiener process defining the Hermite processes is independent of limiting one in case $j \leq n$, note that in case the chaos rank of one component equals $1$ the limit is a fractional Brownian motion. It was shown in \cite{Hairer05} that the filtration between this fBM and the Wiener process defining it are identical, thus, as the fBM is independent of the limit for $j \leq n$ so is the defining Wiener process.
The proclaimed covariances were proved in the sections above, hence, this concludes the proof. 
\end{proof}

\section{Application to homogenization of slow/fast systems}

In this section we give an application of Theorem \ref{theorem-mixed} to a homogenization problem using Young/rough path integration theory.

In the following proof we require the following theorem from rough path theory for details and the corresponding version for Young differential equations we refer to \cite{Friz-Hairer,Friz-Victoir,Lyons94,Lyons-Caruana-Levy}. We denote the space of rough paths of  regularity $\gamma$ by $\FC^{\gamma}$.

\begin{theorem}\label{cty-rough}
	Let $Y_0 \in \R, \gamma \in (\f 1 3, \f 1 2), f \in \C^3_b([0,1],\R) $, and $\X \in \FC^{\gamma}([0,T],\R)$ be given. Then, the differential equation
	\begin{equation}\label{example-sde}
	Y_t = Y_0 + \int_0^t f(Y_s) d\mathbf{X}_s 
	\end{equation}
	has a unique solution which belongs to $\C^{\gamma}\left( [0,1],\R\right)$. Furthermore, the solution map  $\Phi_{f}: ~\R \times  \FC^{\gamma}([0,1], \R)
	\to  \C^{\gamma}([0,1],\R)$, where the first component is the initial condition and the second one the driver $\X$, is continuous.
\end{theorem}

For more details concerning rough path theory we refer to \cite{Friz-Hairer,Lyons-Caruana-Levy,Lyons94,Friz-Victoir}.

\textbf{Proof of Theorem \ref{thm-application}. }
\begin{proof}\label{proof-theorem-application}
Set $X^{\epsilon}_t = \alpha(\epsilon) \int_0^{\f t \epsilon} G(y_s) ds. $ Thus, we may rewrite Equation \ref{limit-eq} as 
$$
d x_t^\epsilon = f(x_t^\epsilon) dX^{\epsilon}_t, 
\qquad  x_0^\epsilon=x_0.
$$
Therefore, in case $1$ the claim follows from Proposition \ref{proposition-joint-wiener-building-blocks-lrd} and the continuity of solutions to Young differential equations, the equivalent of Theorem \ref{cty-rough} in the Young setting, as by assumption $H^*(w) > \f 1 2$.

In case $2$ we need to lift $X^{\epsilon}_t$ to a rough path. However, as we restrict ourselves to $1$ dimensions the rough path lift $\XX^{\epsilon}_{s,t}$ is just given by $\f 1 2 \left( X^{\epsilon}_{s,t} \right)^2$ by symmetry. Although the function $x^2$ is not bounded, due to a truncation argument and our integrability assumptions one can show that  $\f 1 2 \left( X^{\epsilon}_{s,t} \right)^2 \to \f 1 2 \left( X_{s,t} \right)^2 $ in finite dimensional distributions. Again by symmetry the moment bounds from Lemma \ref{lemma-kolmogorv-bound} carry over to $\XX^{\epsilon}_{s,t}$ and we obtain convergence of $\X^{\epsilon} = \left( X^{\epsilon}_t, \XX^{\epsilon}_{s,t} \right)$ to $\left(W_t, \WW_t \right)$, where $W$ denotes a standard Wiener process and $\WW$ its Stratonovich lift, in $\FC^{\gamma}\left( [0,1], \R\right)$ for $\gamma \in (\f 1 3, \f 1 2)$. Therefore, we may conclude the proof with an application of Theorem \ref{cty-rough}. 
\end{proof}

\section{Appendix}

\subsection{Asymptotic Independence}\label{section-ai}
For the proof of Theorem \ref{theorem-mixed}  we need the following Proposition which slightly modifies results from \cite{Nourdin-Rosinski} and \cite{Nourdin-Nualart-Peccati}, which can be found in  \cite{Gehringer-Li-fOU}.

\begin{proposition}\label{proposition-spit-independence}
	Let $q_1 \leq q_2, \dots \leq q_n \leq p_1 \leq p_2, \dots \le p_m$. Let  $f_i^\epsilon \in L^2(\R^{p_i})$, $
	g_i^\epsilon \in L^2(\R^{q_i})$, $F^{\epsilon}=\left(I_{p_1}(f^{\epsilon}_1), \dots , I_{p_m}(f^{\epsilon}_m)\right)$ and  $G^{\epsilon}=\left(I_{q_1}(g^{\epsilon}_1), \dots, I_{q_n}(g^{\epsilon}_n)\right)$. Suppose that 
	for  every $i$, and any $1 \leq r \leq q_i$:
	$$ \Vert f^{\epsilon}_j \otimes_r g^{\epsilon}_i \Vert \to 0.$$
	Then $F^\epsilon \to U$ and $G^{\epsilon} \to V$ weakly imply that $(F^{\epsilon},G^{\epsilon}) \to (U,V)$ jointly, where  $U$ and $V$ are taken to be independent random variables.
\end{proposition}

\subsection{Reduction}
\begin{theorem}[Theorem 3.2 \cite{Billingsley}]\label{billingsley-3.2}
	Given random variables $X^{\epsilon}_M$ and $X^{\epsilon}$ such that $X^{\epsilon}_M$ converges weakly to $X_M$ as $ \epsilon \to 0$, $X_M \to X$ as $M \to \infty$, and 
	$$ \lim_{M \to \infty} \limsup_{\epsilon \to 0 } \Vert X^{\epsilon}_M - X^{\epsilon} \Vert_{L^2(\Omega)} =0,$$
	then, $X^{\epsilon} \to X$ weakly.
\end{theorem}

\begin{theorem}[Cramer-Wold]\label{Cramer-Wold-theorem}
	Given  random variables
	$(X^{1,\epsilon}, \dots, X^{N,\epsilon} )$ and $(X^1, \dots , X^N)$. Then, $(X^{1,\epsilon}, \dots, X^{N,\epsilon} ) \to  (X^1, \dots , X^N)$ in law if and only if for every $(t_1, \dots, t_N) \in \R^N$,
	$$\sum_{j=1}^N t_j X^{j,\epsilon} \to \sum_{j=1}^N t_j X^j,$$
	in law.
\end{theorem}

\newcommand{\etalchar}[1]{$^{#1}$}

\end{document}